\documentclass[psamsfonts]{amsart}

\usepackage{amssymb,amsfonts}
\usepackage[all,arc]{xy}
\usepackage{enumerate}
\usepackage{mathrsfs}
\usepackage{url}
\usepackage{textcomp}

\newtheorem{thm}{Theorem}[section]
\newtheorem{cor}[thm]{Corollary}
\newtheorem{prop}[thm]{Proposition}
\newtheorem{lem}[thm]{Lemma}

\theoremstyle{definition}
\newtheorem{defn}[thm]{Definition}

\newtheorem{exmp}[thm]{Example}

\newtheorem{fact}[thm]{Fact}

\newtheorem{obs}[thm]{Observation}
\newtheorem*{theorem*}{Theorem}

\theoremstyle{remark}
\newtheorem{rem}[thm]{Remark}

\newcommand{\tp}{tp}

\makeatletter
\let\c@equation\c@thm
\makeatother
\numberwithin{equation}{section}

\def\Ind{\setbox0=\hbox{$x$}\kern\wd0\hbox to 0pt{\hss$\mid$\hss} \lower.9\ht0\hbox to 0pt{\hss$\smile$\hss}\kern\wd0} 

\def\Notind{\setbox0=\hbox{$x$}\kern\wd0\hbox to 0pt{\mathchardef \nn=12854\hss$\nn$\kern1.4\wd0\hss}\hbox to 0pt{\hss$\mid$\hss}\lower.9\ht0 \hbox to 0pt{\hss$\smile$\hss}\kern\wd0}

\usepackage{cite}

\title{Invariants related to the tree property}

\author{Nicholas Ramsey}

\date{\today}

\begin{document}

\begin{abstract}
We consider cardinal invariants related to Shelah's model-theoretic tree properties and the relations that obtain between them.  From strong colorings, we construct theories $T$ with $\kappa_{\text{cdt}}(T) > \kappa_{\text{sct}}(T) + \kappa_{\text{inp}}(T)$.  We show that these invariants have distinct structural consequences, by investigating their effect on the decay of saturation in ultrapowers.  This answers some questions of Shelah.  
\end{abstract}

\maketitle
\setcounter{tocdepth}{1}
\tableofcontents

\section{Introduction}
One of the fundamental discoveries in stability theory is that stability is local:  a theory is stable if and only if no formula has the order property.  Among the stable theories, one can obtain a measure of complexity by associating to each theory $T$ its \emph{stability spectrum}, namely, the class of cardinals $\lambda$ such that $T$ is stable in $\lambda$.  A classification of stability spectra was given by Shelah in \cite[Chapter 3]{shelah1990classification}.  Part of this analysis amounts showing that stable theories do not have the tree property and, consequently, that forking satisfies local character.  But a crucial component of that work was studying the approximations to the tree property which can exist in stable theories and what structural consequences they have.  These approximations were measured by a cardinal invariant of the theory called $\kappa(T)$, and Shelah's stability spectrum theorem gives an explicit description of the cardinals in which a given theory $T$ is stable in terms of the cardinality of the set of types in finitely many variables over the empty set and $\kappa(T)$.  Shelah used the definition of $\kappa(T)$ as a template for quantifying the global approximations to other tree properties in introducing the invariants $\kappa_{\text{cdt}}(T)$, $\kappa_{\text{sct}}(T)$, and $\kappa_{\text{inp}}(T)$ (see Definition \ref{patterns} below) which bound approximations to the tree property (TP), the tree property of the first kind (TP$_{1}$), and the tree property of the second kind (TP$_{2}$), respectively.  Eventually, the local condition that a theory does not have the tree property (\emph{simplicity}), and the global condition that $\kappa(T) = \kappa_{cdt}(T) = \aleph_{0}$ (\emph{supersimplicity}) proved to mark substantial dividing lines.  These invariants provide a coarse measure of the complexity of the theory, providing a ``quantitative" description of the patterns that can arise among forking formulas.  They are likely to continue to play a role in the development of a structure theory for tame classes of non-simple theories.  

Motivated by some questions from \cite{shelah1990classification}, we explore which relationships known to hold between the \emph{local} properties TP, TP$_{1}$, and TP$_{2}$ also hold for the \emph{global} invariants $\kappa_{\text{cdt}}(T)$, $\kappa_{\text{sct}}(T)$, and $\kappa_{\text{inp}}(T)$.  In short, we are pursuing the following analogy:
\begin{table}[h!]
  \begin{center}
    \label{tab:table1}
    \begin{tabular}{c|c|c|c}
   local & TP & TP$_{1}$ & TP$_{2}$ \\
      \hline
      global & $\kappa_{\text{cdt}}$ & $\kappa_{\text{sct}}$ & $\kappa_{\text{inp}}$ \\
    \end{tabular}
  \end{center}
\end{table}

\noindent This continues the work done in \cite{ArtemNick}, where, with Artem Chernikov, we considered a global analogue of the following theorem of Shelah:
\begin{theorem*}{\cite[III.7.11]{shelah1990classification}}
For complete theory $T$, $\kappa_{\text{cdt}}(T) = \infty$ and only if $\kappa_{\text{sct}}(T) = \infty$ or $\kappa_{\text{inp}}(T) = \infty$.  That is, $T$ has the tree property if and only if it has the tree property of the first kind or the tree property of the second kind.
\end{theorem*}
\noindent Shelah then asked if $\kappa_{\text{cdt}}(T) = \kappa_{\text{sct}}(T) + \kappa_{\text{inp}}(T)$ in general \cite[Question III.7.14]{shelah1990classification}\footnote{This formulation is somewhat inaccurate.  Shelah defines for $x \in \{\text{cdt},\text{inp},\text{sct}\}$, the cardinal invariant $\kappa r_{x}$, which is the least regular cardinal $\geq \kappa_{x}$.  Shelah's precise question was about the possible equality $\kappa r_{\text{cdt}} = \kappa r_{\text{sct}} + \kappa r_{\text{inp}}$.  For our purposes, we will only need to consider theories in which $\kappa_{x}$ is a successor cardinal, so we will not need to distinguish between these two variations.}.  In \cite{ArtemNick}, we showed that is true under the assumption that $T$ is countable.  For a countable theory $T$, the only possible values of these invariants are $\aleph_{0} ,\aleph_{1}$, and $\infty$\textemdash our proof handled each cardinal separately using a different argument in each case.  Here we consider this question without any hypothesis on the cardinality of $T$, answering the general question negatively (Theorem \ref{first main theorem} below):
\begin{theorem*}
There is a stable theory \(T\) so that \(\kappa_{\text{cdt}}(T) > \kappa_{\text{sct}}(T) + \kappa_{\text{inp}}(T)\).  Moreover, it is consistent with ZFC that for every regular uncountable \(\kappa\), there is a stable theory \(T\) with \(|T| = \kappa\) and \(\kappa_{\text{cdt}}(T) > \kappa_{\text{sct}}(T) + \kappa_{\text{inp}}(T)\).  
\end{theorem*}
  To construct a theory $T$ so that $\kappa_{\text{cdt}}(T) \neq \kappa_{\text{sct}}(T) + \kappa_{\text{inp}}(T)$, we use results on \emph{strong colorings} constructed by Galvin under GCH and later by Shelah in ZFC.  These results show that, at suitable regular cardinals, Ramsey's theorem fails in a particularly dramatic way.  The statement $\kappa_{\text{cdt}}(T) = \kappa_{\text{sct}}(T) + \kappa_{\text{inp}}(T)$ amounts to saying that a certain large global configuration gives rise to another large configuration which is moreover very uniform.  This has the feel of many statements in the partition calculus and we show that, in fact, a coloring $f: [\kappa]^{2} \to 2$ can be used to construct a theory $T^{*}_{\kappa, f}$ such that the existence of a large inp- or sct-patterns relative to $T^{*}_{\kappa,f}$ implies some homogeneity for the coloring $f$.  The  theories built from the strong colorings of Galvin and Shelah, then, furnish ZFC counter-examples to Shelah's question, and also give a consistency result showing that, consistently, for every regular uncountable cardinal $\kappa$, there is a theory $T$ with $|T| = \kappa$ and $\kappa_{\text{cdt}}(T) \neq \kappa_{\text{sct}}(T) + \kappa_{\text{inp}}(T)$.  This suggests that the aforementioned result of \cite{ArtemNick} for countable theories is in some sense the optimal result possible in ZFC.  

Our second theorem is motivated by the following theorem of Shelah:
\begin{theorem*}{\cite[VI.4.7]{shelah1990classification}}\label{saturation}
If $T$ is not simple, $\mathcal{D}$ is a regular ultrafilter over $I$, $M$ is an $|I|^{++}$-saturated model of $T$, then $M^{I}/\mathcal{D}$ is not $|I|^{++}$-compact.  
\end{theorem*}
\noindent In an exercise, Shelah claims that the hypothesis that $T$ is not simple in the above theorem may be replaced by the condition $\kappa_{\text{inp}}(T) > |I|^{+}$ and asks if $\kappa_{\text{cdt}}(T) > |I|^{+}$ suffices \cite[Question VI.4.20]{shelah1990classification}.  We prove, in Corollary \ref{second main theorem} and Theorem \ref{second main theorem part 2} respectively, the following:
\begin{theorem*}
There is a theory $T$ such that $\kappa_{\text{inp}}(T) = \lambda^{++}$ yet for any regular ultrafilter $\mathcal{D}$ on $\lambda$ and $\lambda^{++}$-saturated model $M \models T$, $M^{\lambda}/\mathcal{D}$ is $\lambda^{++}$-saturated.  
\end{theorem*}

\begin{theorem*}
If $\lambda = \lambda^{<\lambda}$ and $\kappa_{\text{sct}}(T) > \lambda^{+}$, $M$ is an $\lambda^{++}$-saturated model of $T$ and $\mathcal{D}$ is a regular ultrafilter over $\lambda$, then $M^{\lambda}/\mathcal{D}$ is not $\lambda^{++}$-compact. 
\end{theorem*}
\noindent The first of these results contradicts Shelah's Exercise VI.4.19 and \emph{a fortiori} answers Question VI.4.20 negatively.  Although $\kappa_{\text{inp}}(T) > |I|^{+}$ and hence $\kappa_{\text{cdt}}(T) > |I|^{+}$ do not suffice to guarantee a loss of saturation in the ultrapower, one can ask if $\kappa_{\text{sct}}(T) > |I|^{+}$ does suffice.  Shelah's original argument for Theorem \ref{saturation} does not generalize, but fortunately a recent new proof due to Malliaris and Shelah \cite{Malliaris:2012aa} does and we point out in the second of these two theorems how the revised question can be answered, modulo a mild set-theoretic hypothesis, by an easy and direct adaptation of their argument. These results suggest that the rough-scale asymptotic structure revealed by studying the $\lambda^{++}$-compactness of ultrapowers on $\lambda$ is global in nature and differs from the picture suggested by the local case considered by Shelah.  

In order to construct these examples, it is necessary to build a theory capable of coding a complicated strong coloring yet simple enough that the invariants are still computable.  This was accomplished by a method inspired by Medvedev's $\mathbb{Q}$ACFA construction \cite{Medvedev:2015aa}, realizing the theory as a union of theories in a system of finite reducts each of which is the theory of a Fra\"iss\'e limit.  The theories in the finite reducts are $\aleph_{0}$-categorical and eliminate quantifiers and one may apply the $\Delta$-system lemma to the finite reducts arising in global configurations.  Altogether, this makes computing the invariants tractable.  

\textbf{Acknowledgements:}  This is work done as part of our dissertation under the supervision of Thomas Scanlon.  We would additionally like to acknowledge very helpful input from Artem Chernikov, Leo Harrington, Alex Kruckman, and Maryanthe Malliaris, as well as Assaf Rinot, from whom we first learned of Galvin's work on strong colorings.  Finally we would like to thank the anonymous referee for more than one especially thorough reading which did a great deal to improve this paper.

\section{Preliminaries}

\subsection{Notions from Classification Theory}

For the most part, we follow standard model-theoretic notation.  We may write $x$ or $a$ to denote a tuple of variables or elements, which may not have length 1.  If $x$ is a tuple of variables we write $l(x)$ to denote its length and for each $l < l(x)$, we write $(x)_{l}$ to denote the $l$th coordinate of $x$.  If $\varphi(x)$ is a formula and $t \in \{0,1\}$, we write $\varphi(x)^{t}$ to denote $\varphi(x)$ if $t = 1$ and $\neg \varphi(x)$ if $t = 0$.  

In the following definitions, we will refer to collections of tuples indexed by arrays and trees.  For cardinals $\kappa$ and $\lambda$, we use the notation $\unlhd$, $<_{lex}$, $\wedge$, and $\perp$ to refer to the tree partial order, the lexicographic order, the binary meet function, and the relation of incomparability on $\kappa^{<\lambda}$, respectively.  Given an element $\eta \in \kappa^{<\lambda}$, we write $l(\eta)$ to denote the length of $\eta$\textemdash that is, the unique $\alpha < \lambda$ such that $\eta \in \kappa^{\alpha}$\textemdash and if $l(\eta)\geq \beta$, we write $\eta | \beta$ for the unique $\nu \unlhd \eta$ with $l(\nu) = \beta$.  

\begin{defn} \label{patterns}  
\cite[Definitions III.7.2, III.7.3, III.7.5]{shelah1990classification}
\begin{enumerate}
\item A \emph{cdt-pattern of height} \(\kappa\) is a sequence of formulas \(\varphi_{i}(x;y_{i})\) (\(i < \kappa, i \text{ successor}\)) and numbers \(n_{i} < \omega\), and a tree of tuples \((a_{\eta})_{\eta \in \omega^{<\kappa}}\) for which 
\begin{enumerate}
\item \(p_{\eta} = \{\varphi_{i}(x;a_{\eta | i}) : i \text{ successor }, i < \kappa\}\) is consistent for \(\eta \in \omega^{\kappa}\).
\item \(\{\varphi_{i} (x;a_{\eta \frown \langle \alpha \rangle}) : \alpha < \omega , i = l(\eta) + 1\}\) is \(n_{i}\)-inconsistent.
\end{enumerate}
\item An \emph{inp-pattern of height} \(\kappa\) is a sequence of formulas \(\varphi_{i}(x;y_{i})\) \((i < \kappa)\), sequences \((a_{i,\alpha}: \alpha < \omega)\), and numbers \(n_{i} <\omega\) such that 
\begin{enumerate}
\item For any \(\eta \in \omega^{\kappa}\), \(\{ \varphi_{i}(x;a_{i,\eta(i)}) : i < \kappa\}\) is consistent.
\item For any \(i < \kappa\), \(\{\varphi_{i}(x;a_{i,\alpha}) : \alpha < \omega\}\) is \(n_{i}\)-inconsistent.  
\end{enumerate}
\item An \emph{sct-pattern of height} \(\kappa\) is a sequence of formulas \(\varphi_{i}(x;y_{i})\) \((i < \kappa)\) and a tree of tuples \((a_{\eta})_{\eta \in \omega^{<\kappa}}\) such that 
\begin{enumerate}
\item For every \(\eta \in \omega^{\kappa}\), \(\{\varphi_{\alpha}(x;a_{\eta | \alpha}) : 0 < \alpha < \kappa, \alpha \text{ successor}\}\) is consistent.
\item If \(\eta \in \omega^{\alpha}\), \(\nu \in \omega^{\beta}\), \(\alpha, \beta\) are successors, and \(\nu \perp \eta\) then \(\{\varphi_{\alpha}(x;a_{\eta}), \varphi_{\beta}(x;a_{\nu})\}\) are inconsistent.  
\end{enumerate}
\item For \(X \in \{\text{cdt}, \text{sct}, \text{inp}\}\), we define \(\kappa_{X}^{n}(T)\) be the first cardinal \(\kappa\) such that there is no \(X\)-pattern of height \(\kappa\) in \(n\) free variables.  We define \(\kappa_{X}(T) = \sup_{n} \{\kappa_{X}^{n}\}\).  
\end{enumerate}
\end{defn}

When introducing these definitions, Shelah notes that cdt stands for ``contradictory types" and inp stands for ``independent partitions."  He does not explain the meaning of sct, but presumably it is intended to abbreviate something like ``strongly contradictory types".

\begin{fact} \label{easy inequalities} \cite[Observation 3.1]{ArtemNick} Suppose $T$ is a complete theory in the language $L$.
\begin{enumerate}
\item If $T$ is stable, then $\kappa_{\mathrm{cdt}}(T) \leq |L|^{+}$.
\item $\kappa_{\mathrm{sct}}(T) \leq \kappa_{\mathrm{cdt}}(T)$ and $\kappa_{\mathrm{inp}}(T) \leq \kappa_{\mathrm{cdt}}(T)$.  	
\end{enumerate}	
\end{fact}

\begin{exmp}
Fix a regular uncountable cardinal \(\kappa\) and let \(L = \langle E_{\alpha} : \alpha < \kappa \rangle\) be a language consisting of $\kappa$ many binary relations.  Let $T_{\text{sct}}$ be the model companion of the $L$-theory asserting that each $E_{\alpha}$ is an equivalence relation and $\alpha < \beta$ implies $E_{\beta}$ refines $E_{\alpha}$.  Let $T_{\text{inp}}$ be the model companion of the $L$-theory which only asserts that each $E_{\alpha}$ is an equivalence relation.  In other words, $T_{\text{sct}}$ is the generic theory of $\kappa$ refining equivalence relations and $T_{\text{inp}}$ is the generic theory of $\kappa$ independent equivalence relations.  Now $\kappa_{\text{cdt}}(T_{\text{sct}}) = \kappa_{\text{cdt}}(T_{\text{sct}}) = \kappa^{+}\), and further \(\kappa_{\text{sct}}(T_{\text{sct}}) = \kappa_{\text{inp}}(T_{\text{inp}}) = \kappa^{+}\).  However, we have \(\kappa_{\text{inp}}(T_{\text{sct}}) = \aleph_{0}\) and \(\kappa_{\text{sct}}(T_{\text{inp}}) = \aleph_{1}\).  

Computing each of the invariants is straightforward using quantifier elimination for $T_{\text{inp}}$ and $T_{\text{sct}}$ with the exception of $\kappa_{\text{sct}}(T_{\text{inp}}) = \aleph_{1}$.  The fact that $\kappa_{\text{cdt}}(T_{\text{inp}}) \geq \aleph_{1}$ implies that $\kappa_{\text{sct}}(T_{\text{inp}}) \geq \aleph_{1}$ by \cite[Proposition 3.14]{ArtemNick}.  If $\kappa_{\text{sct}}(T_{\text{inp}}) > \aleph_{1}$ then there is an sct-pattern $(\varphi_{\alpha}(x;y_{\alpha}) : \alpha < \omega_{1})$, $(a_{\eta})_{\eta \in \omega^{<\omega_{1}}}$. Let $w_{\alpha}$ be the finite set of indices $\beta$ such that the symbol $E_{\beta}$ appears in $\varphi_{\alpha}(x;y_{\alpha})$.  After passing to an sct-pattern of the same size, we may assume that the $w_{\alpha}$ form a $\Delta$-system (see Fact \ref{delta-system lemma} below), using that $\kappa$ is regular and uncountable.  Now it is easy to check using quantifier elimination for $T_{\text{sct}}$ that there are incomparable $\eta \in \omega^{\alpha}, \nu \in \omega^{\beta}$ for some $\alpha,\beta < \omega_{1}$ such that $\{\varphi_{\alpha}(x;a_{\eta}), \varphi_{\beta}(x;a_{\nu})\}$ is consistent, a contradiction.\end{exmp}
%
%
%
%

The following simple observation will be useful:
\begin{lem} \label{no equalities}
Suppose $\kappa$ is an infinite cardinal.  
\begin{enumerate}
\item Suppose $(\varphi_{\alpha}(x;y_{\alpha}) : \alpha < \kappa)$, $(a_{\alpha,i})_{\alpha < \kappa, i < \omega}$, $(k_{\alpha})_{\alpha < \kappa}$ is an inp-pattern with $l(x) = 1$.  Then each formula $\varphi_{\alpha}(x;a_{\alpha,i})$ is non-algebraic.  
\item Suppose $(\varphi_{\alpha}(x;y_{\alpha}) : \alpha < \kappa)$, $(a_{\eta})_{\eta \in \omega^{<\kappa}}$ is an sct-pattern such that $l(x)$ is minimal among sct-patterns of height $\kappa$ modulo $T$.  Then no formula $\varphi_{\alpha}(x;a_{\eta})$ with $\eta \in \omega^{\alpha}$ implies $(x)_{l} = c$ for some $l < l(x)$ and parameter $c$.  
\end{enumerate}
\end{lem}

\begin{proof}
(1)  Given any $\alpha < \kappa$ and $i < \omega$, we may, for each $j < \omega$, choose a realization $c_{j} \models \{\varphi_{\alpha}(x;a_{\alpha,i}), \varphi_{\alpha+1}(x;a_{\alpha+1,j})\}$, which is is consistent by the definition of an inp-pattern.  Since $\{\varphi_{\alpha+1}(x;a_{\alpha+1,j}) : j < \omega\}$ is $k_{\alpha+1}$-inconsistent, each $c_{j}$ can realize at most $k_{\alpha+1}-1$ many formulas in this set, so $\{c_{j} : j < \omega\}$ must be an infinite set of realizations of $\varphi_{\alpha}(x;a_{\alpha,i})$, which shows $\varphi_{\alpha}(x;a_{\alpha})$ is non-algebraic.

(2)  Suppose not, so there are $\alpha < \kappa$, $\eta \in \omega^{\alpha}$, and $l < l(x)$ so that $\varphi_{\alpha}(x;a_{\eta}) \vdash (x)_{l} = c$ for some parameter $c$; without loss of generality $l = l(x) - 1$.  If $l(x) = 1$, then it follows from the fact that $\{\varphi_{\alpha}(x;a_{\eta}),\varphi_{\alpha+1}(x;a_{\eta \frown \langle i \rangle})\}$ is consistent for each $i < \omega$ that $c \models \{\varphi_{\alpha+1}(x;a_{\eta \frown \langle i \rangle}) : i < \omega\}$, contradicting the fact that this set of formulas is $2$-inconsistent.  On the other hand, if $l  > 1$, we will let $x' = (x_{0}, \ldots, x_{l-2})$, so that $x = (x',x_{l-1})$ and let $b_{\nu} = (c,a_{\eta \frown \nu})$ for all $\nu \in \omega^{<\kappa}$.  Finally, we set $\psi_{\beta}(x';z_{\beta}) = \varphi_{\alpha + \beta}(x';x_{l-1},y_{\alpha+\beta})$.  Since for any $\nu \in \omega^{\kappa}$, $\{\varphi_{\alpha + \beta}(x;a_{\eta \frown (\nu | \beta)}) : \beta < \kappa\}$ is consistent and any realization will be of the form $(c',c)$ for some $c'$, it follows that $\{\psi_{\beta}(x';b_{\nu | \beta}) : \beta < \kappa\}$ is consistent.  The inconsistency requirement is immediate so it follows that $(\psi_{\beta}(x';z_{\beta}))_{\beta < \kappa}$, $(b_{\eta})_{\eta \in \omega^{<\kappa}}$ is an sct-pattern of height $\kappa$ in fewer than $l(x)$ variables, contradicting the minimality of $l(x)$.  
\end{proof}

\begin{rem}
Note that by \cite[Corollary 2.9]{ChernikovNTP2}, if $T$ has an inp-pattern of height $\kappa$, then there is also an inp-pattern of height $\kappa$ in a single free variable, so the hypothesis in (1) that $l(x) = 1$ is equivalent to the requirement that $l(x)$ be minimal among inp-patterns of height $\kappa$.  
\end{rem}

In order to simplify many of the arguments below, it will be useful to work with indiscernible trees and arrays.  Define a language \(L_{s,\lambda} = \{\vartriangleleft, \wedge, <_{lex}, P_{\alpha} : \alpha < \lambda\}\) where \(\lambda\) is a cardinal.  We may view the tree \(\kappa^{<\lambda}\) as an \(L_{s,\lambda}\)-structure in a natural way, giving \(\vartriangleleft\), \(\wedge\), and \(<_{lex}\) their eponymous interpretations, and interpreting \(P_{\alpha}\) as a predicate which identifies the \(\alpha\)th level.  Note that we may define the relation $\eta \perp \nu$ in this language by $\neg (\eta \unlhd \nu) \wedge \neg (\nu \unlhd \eta)$.  See \cite{ArtemNick} and \cite{KimKimScow} for a detailed treatment.

\begin{defn}
\text{ }
\begin{enumerate}
\item We say \((a_{\eta})_{\eta \in \kappa^{<\lambda}}\) is an \(s\)\emph{-indiscernible tree over A} if
\[
\text{qftp}_{L_{s,\lambda}}(\eta_{0}, \ldots, \eta_{n-1}) = \text{qftp}_{L_{s,\lambda}}(\nu_{0}, \ldots, \nu_{n-1})
\] 
implies \(\text{tp}(a_{\eta_{0}}, \ldots, a_{\eta_{n-1}}/A) = \text{tp}(a_{\nu_{0}}, \ldots, a_{\nu_{n-1}}/A)\).  
\item We say \((a_{\alpha,i})_{\alpha < \kappa, i < \omega}\) is a \emph{mutually indiscernible array} over $A$ if, for all $\alpha < \kappa$, $(a_{\alpha, i})_{i < \omega}$ is a sequence indiscernible over $A \cup \{a_{\beta,j} : \beta < \kappa, \beta \neq \alpha, j < \omega\}$.
\end{enumerate}
\end{defn}

\begin{fact} \cite[Theorem 4.3]{KimKimScow} \label{s-indiscernible extraction}
Given a collection of tuples $(a_{\eta})_{\eta \in \omega^{<\omega}}$, there is $(b_{\eta})_{\eta \in \omega^{<\omega}}$ which is $s$-indiscernible and \emph{locally based} on $(a_{\eta})_{\eta \in \omega^{<\omega}}$, that is, given any $\overline{\eta} = (\eta_{0},\ldots, \eta_{k-1}) \in \omega^{<\omega}$ and $\varphi(x_{0},\ldots, x_{n-1})$ such that $\models \varphi(b_{\eta_{0}},\ldots, b_{\eta_{k-1}})$, there is $\overline{\nu} = (\nu_{0},\ldots, \nu_{n-1}) \in \omega^{<\omega}$ with $\mathrm{qftp}_{L_{s,\omega}}(\overline{\eta}) = \mathrm{qftp}_{L_{s,\omega}}(\overline{\nu})$ and $\models \varphi(a_{\nu_{0}},\ldots, a_{\nu_{n-1}})$.  
\end{fact}

\begin{fact} \cite[Lemma 1.2(2)]{ChernikovNTP2} \label{artem's lemma}
Let $(a_{\alpha,i})_{\alpha < n, i < \omega}$ be an array of parameters.  Given a finite set of formulas $\Delta$ and $N < \omega$, we can find, for each $\alpha < n$, $i_{\alpha,0} < i_{\alpha,1} < \ldots < i_{\alpha,N-1}$ so that $(a_{\alpha,i_{\alpha,j}})_{\alpha < n, j < N}$ is $\Delta$-mutually indiscernible array\textemdash i.e. for all $\alpha < n$, $(a_{\alpha,i_{\alpha,j}})_{j < N}$ is $\Delta$-indiscernible over $\{a_{\beta,i_{\beta,j}} : \beta \neq \alpha, j < N\}$.  
\end{fact}

\begin{fact} \label{s-IndiscTreeProp} \cite[Lemma 2.2]{ArtemNick}
Let $(a_\eta : \eta \in \kappa^{<\lambda})$ be a tree s-indiscernible over a set of parameters $C$.
\begin{enumerate}
\item All paths have the same type over $C$: for any $\eta, \nu \in \kappa^{\lambda}$, $\tp((a_{\eta | \alpha})_{\alpha < \lambda}/C) = \tp((a_{\nu|\alpha})_{\alpha < \lambda}/C)$.
\item Suppose $\{\eta_{\alpha} : \alpha < \gamma\} \subseteq \kappa^{<\lambda}$ satisfies $\eta_{\alpha} \perp \eta_{\alpha'}$ whenever $\alpha \neq \alpha'$.  Then the array $(b_{\alpha, \beta})_{\alpha < \gamma, \beta < \kappa}$ defined by 
$$
b_{\alpha, \beta} = a_{\eta_{\alpha} \frown \langle \beta \rangle}
$$
is mutually indiscernible over $C$.  
\end{enumerate}
\end{fact}
%

Parts (1) and (2) of the following lemma are essentially \cite[Lemma 2.2]{ChernikovNTP2} and \cite[Lemma 3.1(1)]{ArtemNick}, respectively, but we sketch the argument in order to point out that, from a inp- or sct-pattern of height $\kappa$, we can find one with appropriately indiscernible parameters, leaving the formulas fixed.  

\begin{lem} \label{witness}
\begin{enumerate}
\item If there is an inp-pattern $(\varphi_{\alpha}(x;y_{\alpha}) : \alpha < \kappa)$, $(a_{\alpha,i})_{\alpha < \kappa, i < \omega}$, $(k_{\alpha})_{\alpha < \omega}$ of height $\kappa$ modulo $T$, then there is an inp-pattern $(\varphi_{\alpha}(x;y_{\alpha}) : \alpha < \kappa)\), $(a'_{\alpha, i})_{\alpha < \kappa, i < \omega}$, $(k_{\alpha})_{\alpha < \kappa}$ such that $(a'_{\alpha, i})_{\alpha < \kappa, i < \omega}$ is a mutually indiscernible array.  
\item If there is an sct-pattern (cdt-pattern) of height $\kappa$ modulo $T$, then there is an sct-pattern (cdt-pattern) $\varphi_{\alpha}(x;y_{\alpha})$, $(a_{\eta})_{\eta \in \omega^{<\kappa}}$ such that $(a_{\eta})_{\eta \in \omega^{<\kappa}}$ is an $s$-indiscernible tree. 
\end{enumerate}
\end{lem}

\begin{proof}
(1)  Given an inp-pattern $(\varphi_{\alpha}(x;y_{\alpha}) : \alpha < \kappa)$, $(a_{\alpha,i})_{\alpha < \kappa, i < \omega}$, $(k_{\alpha})_{\alpha < \omega}$, let $\Gamma(z_{\alpha,i}  : \alpha < \kappa, i < \omega)$ be a partial type that naturally expresses the following:
\begin{itemize}
\item $(z_{\alpha,i})_{\alpha < \kappa, i < \omega}$ is a mutually indiscernible array.
\item $\{\varphi_{\alpha}(x;z_{\alpha,i}) : i < \omega\}$ is $k_{\alpha}$-inconsistent.
\item For every $f: \kappa \to \omega$, $\{\varphi_{\alpha}(x;z_{\alpha,f(\alpha)}) : \alpha < \kappa\}$ is consistent.
\end{itemize}
By Lemma \ref{artem's lemma}, any finite subset of $\Gamma$ this partial type can be satisfied by an array from $(a_{\alpha,i})_{\alpha < \kappa, i < \omega}$ and therefore $\Gamma$ is consistent by compactness.  A realization $(a'_{\alpha,i})_{\alpha < \kappa, i < \omega}$ yields the desired inp-pattern.

(2)  is entirely similar:  given an sct-pattern $\varphi_{\alpha}(x;y_{\alpha})$, $(a_{\eta})_{\eta \in \omega^{<\kappa}}$, apply Fact \ref{s-indiscernible extraction} and compactness to obtain $(b_{\eta})_{\eta \in \omega^{<\kappa}}$, which is $s$-indiscernible and has the property that for any formula $\varphi(x_{0},\ldots, x_{n-1})$ and $\overline{\eta} = (\eta_{0},\ldots, \eta_{n-1}) \in \omega^{<\kappa}$, if $\varphi(b_{\eta_{0}},\ldots, b_{\eta_{n-1}})$, there is $\overline{\nu} = (\nu_{0},\ldots, \nu_{n-1})$ with $\mathrm{qftp}_{L_{s,\kappa}}(\overline{\eta}) = \mathrm{qftp}_{L_{s,\kappa}}(\overline{\nu})$ such that $\varphi(a_{\nu_{0}},\ldots, a_{\nu_{n-1}})$.  From this property, it easily follows that, for all $\eta \in \omega^{\alpha}$, $\{\varphi_{\alpha+1}(x;a_{\eta \frown \langle i \rangle}) : i < \omega \}$ is $k_{\alpha+1}$-inconsistent and, for all $\eta \in \omega^{\kappa}$, $\{\varphi_{\alpha}(x;a_{\eta | \alpha}) : \alpha < \kappa\}$ is consistent.  Therefore $(\varphi_{\alpha}(x;y_{\alpha}) : \alpha < \kappa)$, $(b_{\eta})_{\eta \in \omega^{<\kappa}}$ is the desired sct-pattern.    
\end{proof}

\subsection{Fra\"iss\'e Theory}

We will recall some basic facts from Fra\"iss\'e theory, from \cite[Section 7.1]{hodges1993model}.  Let \(L\) be a finite language and let \(\mathbb{K}\) be a non-empty finite or countable set of finitely generated \(L\)-structures which has HP, JEP, and AP.  Such a class $\mathbb{K}$ is called a \emph{Fra\"iss\'e class}.  Then there is an \(L\)-structure \(D\), unique up to isomorphism, such that \(D\) has cardinality \(\leq \aleph_{0}\), \(\mathbb{K}\) is the age of \(D\), and \(D\) is ultrahomogeneous.  We call $D$ the \emph{Fra\"iss\'e limit} of $\mathbb{K}$, which we sometimes denote $\text{Flim}(\mathbb{K})$.  Given a subset $A$ of the $L$-structure $C$, we write $\langle A \rangle^{C}_{L}$ for the $L$-substructure of $C$ generated by $A$.  We say that $\mathbb{K}$ is \emph{uniformly locally finite} if there is a function $g: \omega \to \omega$ such that a structure in $\mathbb{K}$ generated by $n$ elements has cardinality at most $g(n)$.  If \(\mathbb{K}\) is a countable uniformly locally finite set of finitely generated \(L\)-structures and $T = \text{Th}(D)$, then \(T\) is \(\aleph_{0}\)-categorical and has quantifier elimination.  

The following equivalent formulation of ultrahomogeneity is well-known, see, e.g., \cite[Proposition 2.3]{KPT}:

\begin{fact}
Let \(A\) be a countable structure.  Then \(A\) is ultrahomogeneous if and only if it satisfies the following extension property:  if \(B,C\) are finitely generated and can be embedded into \(A\), \(f: B \to A\), \(g: B \to C\) are embeddings then there is an embedding \(h: C \to A\) such that \(h \circ g = f\).  
\end{fact}

The following is a straight-forward generalization of \cite[Proposition 5.2]{KPT}:

\begin{lem}\label{KPTredux}
Suppose \(L \subseteq L'\), and \(\mathbb{K}\) is a Fra\"iss\'e class of \(L\)-structures and \(\mathbb{K}'\) is a Fra\"iss\'e class of \(L'\)-structures satisfying the following two conditions:
\begin{enumerate}
\item \(A \in \mathbb{K}\) if and only if there is a \(D' \in \mathbb{K}'\) such that \(A\) is an $L$-substructure of \(D' \upharpoonright L\). 
\item If \(A,B \in \mathbb{K}\), \(\pi: A \to B\) is an \(L\)-embedding, and \(C \in \mathbb{K}'\) with \(C = \langle A \rangle^{C}_{L'}\), then there is a \(D \in \mathbb{K}'\), such that $B$ is an $L$-substructure of $D\upharpoonright L$, and an \(L'\)-embedding \(\tilde{\pi}: C \to D\) extending \(\pi\).  
\end{enumerate}
Then \(\mathrm{Flim}(\mathbb{K}') \upharpoonright L  = \mathrm{Flim}(\mathbb{K})\).  
\end{lem}

\begin{proof}
Let \(F' = \text{Flim}(\mathbb{K}')\) and suppose \(F = F' \upharpoonright L\).  Fix \(A_{0}, B_{0} \in \mathbb{K}\) and an \(L\)-embedding \(\pi: A_{0} \to B_{0}\).  Suppose \(\varphi: A_{0} \to F\) is an \(L\)-embedding.  Let \(E = \langle \varphi(A_{0}) \rangle^{F'}_{L'}\).  Up to isomorphism over \(A_{0}\), there is a unique \(C \in \mathbb{K}'\) containing \(A_{0}\) such that \(C = \langle A_{0} \rangle^{C}_{L'}\) and \(\tilde{\varphi} : C \to F'\) is an \(L'\)-embedding extending \(\varphi\) with \(E = \tilde{\varphi}(C)\), since given another such $C'$ and $\tilde{\varphi}' :C' \to F'$, we have $\tilde{\varphi}'^{-1} \circ \tilde{\varphi} : C \to C'$ is an $L'$-isomorphism which is the identity on $A_{0}$.  By (2), there is some \(D \in \mathbb{K}'\) with \(B_{0} \subseteq D \upharpoonright L\) and and there is an \(L'\)-embedding \(\tilde{\pi}: C \to D\) extending \(\pi\).  By the extension property for \(F'\), there is an \(L'\)-embedding \(\psi: D \to F'\) such that \(\psi \circ \tilde{\pi} = \tilde{\varphi}\) and hence \(\psi \circ \pi = \varphi\).  As \(\psi \upharpoonright B_{0}\) is an \(L\)-embedding, this shows the extension property for \(F\).  So \(F\) is ultrahomogeneous, and \(\text{Age}(F) = \mathbb{K}\) by (1) so \(F \cong \text{Flim}(\mathbb{K})\), which completes the proof.  
\end{proof}
%
%

\subsection{Strong Colorings}
\begin{defn}  \cite[Definition A.1.2]{Sh:g}
Given cardinals $\lambda, \mu, \theta,$ and $\chi$, we write \(\text{Pr}_{1}(\lambda, \mu, \theta, \chi)\) for the assertion:  there is a coloring \(c: [\lambda]^{2} \to \theta\) such that for any \(A \subseteq [\lambda]^{<\chi}\) of size \(\mu\) consisting of pairwise disjoint subsets of \(\lambda\) and any color \(\gamma < \theta\) there are \(a,b \in A\) with \(\max(a) < \min(b)\) with \(c(\{\alpha, \beta\}) = \gamma\) for all \(\alpha \in a\), \(\beta \in b\).  
\end{defn}

Note, for example, that $\text{Pr}_{1}(\lambda, \lambda, 2, 2)$ holds if and only if $\lambda \not\to (\lambda)^{2}_{2}$ - i.e. $\lambda$ is not weakly compact.  

\begin{obs} \label{monotonicity}
For fixed $\lambda$, if $\mu \leq \mu'$, $\theta' \leq \theta$, $\chi' \leq \chi$, then 
$$
\text{Pr}_{1}(\lambda, \mu, \theta, \chi) \implies \text{Pr}_{1}(\lambda, \mu', \theta', \chi').
$$
\end{obs}

\begin{proof}
Fix $c: [\lambda]^{2} \to \theta$ witnessing $\text{Pr}_{1}(\lambda, \mu, \theta, \chi)$.  Define a new coloring $c': [\lambda]^{2} \to \theta'$ by $c'(\{\alpha, \beta\}) = c(\{\alpha, \beta\})$ if $c(\{\alpha, \beta\}) < \theta'$ and $c'(\{\alpha, \beta\}) = 0$ otherwise.  Now suppose $A \subseteq [\lambda]^{< \chi'}$ is a family of pairwise disjoint sets with $|A| \geq \mu'$.  Then, in particular, $A \subseteq [\lambda]^{<\chi}$ and $|A| \geq \mu$ so for any $\gamma < \theta'$, as $\gamma < \theta$, there are $a,b \in A$ with $\text{max}(a) < \text{min}(b)$ with $c'(\{\alpha,\beta\}) = c(\{\alpha,\beta\}) = \gamma$ for all $\alpha \in a$, $\beta \in b$, using $\text{Pr}_{1}(\lambda, \mu, \theta, \chi)$ and the definition of $c'$.  This shows that $c'$ witnesses $\text{Pr}_{1}(\lambda, \mu',\theta',\chi')$.  
\end{proof}

In the arguments that follow, we will only make use of instances of $\mathrm{Pr}_{1}(\lambda^{+},\lambda^{+}, 2, \aleph_{0})$, which we will obtain from stronger results of Galvin and of Shelah, using Observation \ref{monotonicity}.  Galvin proved \(\text{Pr}_{1}\) holds in some form for arbitrary successor cardinals from instances of GCH.  Considerably later, Shelah proved that $\text{Pr}_{1}$ holds in a strong form for the double-successors of arbitrary regular cardinals in ZFC.  


\begin{fact}\cite[Conclusion 4.2]{Sh:572} \label{ShelahPr} The principle \(\text{Pr}_{1}(\lambda^{++}, \lambda^{++}, \lambda^{++}, \lambda)\) holds for every regular cardinal \(\lambda\).  
\end{fact}

The above theorem of Shelah suffices to produce a ZFC counterexample to the equality $\kappa_{\mathrm{cdt}}(T) = \kappa_{\mathrm{inp}}(T) + \kappa_{\mathrm{sct}}(T)$, but we will need Galvin's result on arbitrary successor cardinals in order to get the consistency result contained in Theorem \ref{first main theorem}.  Unfortunately, Galvin's result is only implicit in \cite[Lemma 4.1]{Galvin80} in a certain construction, and the argument there refers to earlier sections of his paper.  So, following a suggestion of the referee, we have opted for providing a self-contained proof.  The argument below merely consolidates Galvin's argument in \cite[Lemma 4.1]{Galvin80} and recasts it in Shelah's $\mathrm{Pr}_{1}$ notation, adding no new ideas.  

It will be useful to introduce the following notation:  given sets $X$ and $Y$, let $X \otimes Y = \{\{x,y\} : x \in X,y \in Y\}$.      	

\begin{lem} \cite[Lemma 3.1]{Galvin80} \label{matrix lemma} 
Let $\lambda$ be an infinite cardinal and $A$ be a set.  Suppose that, for each $\rho < \lambda$, we have a set $I_{\rho}$ with $|I_{\rho}| = \lambda$ and finite sets $E^{\xi}_{\rho} \subseteq A$ $(\xi \in I_{\rho})$ so that for any $a \in A$, $|\{\xi \in I_{\rho} : a \in E^{\xi}_{\rho}\}| < \aleph_{0}$.  Then there are pairwise disjoint sets $(A_{\nu} : \nu < \lambda)$ so that for all $\nu < \lambda$ and $\rho < \lambda$
$$
|\{\xi  \in I_{\rho} : E^{\xi}_{\rho} \subseteq A_{\nu}\}| = \lambda.
$$
\end{lem}

\begin{proof}
Identify $I_{\rho}$ with $\lambda$ for all $\rho$ and let $<^{*}$ be a well-ordering of $\lambda \times \lambda$ in order-type $\lambda$.  By recursion on $(\lambda \times \lambda, <^{*})$, define $(\xi_{(\alpha, \beta)} : (\alpha, \beta) \in \lambda \times \lambda)$ as follows:  if $(\xi_{(\gamma, \delta)} : (\gamma, \delta) <^{*} (\alpha, \beta))$ has been defined, choose $\xi_{(\alpha, \beta)}$ to be the least $\xi \in I_{\alpha}$ so that 
$$
E^{\xi}_{\alpha} \cap \left( \bigcup_{\substack{(\gamma, \delta) <^{*} (\alpha, \beta) \\ \delta \neq \beta}} E^{\xi_{(\gamma, \delta)}}_{\gamma} \right) = \emptyset.  
$$ 
There is such a $\xi$ by the pigeonhole principle, given our assumption that $|\{\xi \in I_{\rho} : a \in E^{\xi}_{\rho}\}| < \aleph_{0}$ for all $a \in A$.  Now define the sequence of sets $(A_{\nu} : \nu < \lambda)$ by
$$
A_{\nu} = \bigcup_{\alpha < \lambda} E_{\nu}^{\xi_{(\alpha, \nu)}}.
$$
It is easy to check that this satisfies the requirements.  
\end{proof}

\begin{thm}  \cite[Lemma 4.1]{Galvin80} \label{GalvinPr}
If \(\lambda\) is an infinite cardinal and \(2^{\lambda} = \lambda^{+}\), then \(\text{Pr}_{1}(\lambda^{+}, \lambda^{+}, \lambda^{+}, \aleph_{0})\).
\end{thm}

\begin{proof}
Let $\langle \overline{B}_{\gamma} : \gamma < \lambda^{+} \rangle$ enumerate all $\lambda$-sequences $\overline{B} = \langle B_{\xi} : \xi < \lambda\rangle$ of pairwise disjoint finite subsets of $\lambda^{+}$.  This is possible as $2^{\lambda} = \lambda^{+}$.    

\textbf{Claim 1}:  There is a sequence of pairwise disjoint sets $\langle K_{\nu} : \nu < \lambda^{+} \rangle$ so that, for all $\nu < \lambda^{+}$, $K_{\nu} \subseteq [\lambda^{+}]^{2}$ and, for all $\alpha < \lambda^{+}$, we have $(A)$ implies $(B)$, where:
\begin{enumerate}
\item[(A)] $\gamma < \alpha$, $\bigcup_{\xi < \lambda} B_{\gamma, \xi} \subseteq \alpha$, $X \in [\alpha]^{<\aleph_{0}}$, and $|\{\xi : B_{\gamma,\xi} \otimes X \subseteq K_{\nu}\}| = \lambda$.
\item[(B)] $|\{ \xi : B_{\gamma,\xi} \otimes (X \cup \{\alpha\}) \subseteq K_{\nu} \}| = \lambda$.  
\end{enumerate}
\emph{Proof of claim}:  
By induction on $\alpha < \lambda^{+}$, we will construct for every $\nu < \lambda$, a set $K_{\nu}(\alpha) \subseteq \alpha$ and define $K_{\nu} = \{\{\beta, \alpha\} : \alpha < \lambda^{+}, \beta \in K_{\nu}(\alpha)\}$.  We will define the sets $K_{\nu}(\alpha)$ to be pairwise disjoint and so that:
\begin{enumerate}
\item[(*)] Whenever $\gamma < \alpha$, $\bigcup_{\xi < \lambda} B_{\gamma,\xi} \subseteq \alpha$, $X \in [\alpha]^{<\aleph_{0}}$, and $|\{\xi : B_{\gamma,\xi} \otimes X \subseteq K_{\nu}\}| = \lambda$, then $|\{\xi : B_{\gamma,\xi} \otimes (X \cup \{\alpha\}) \subseteq K_{\nu}\}| = \lambda$.  
\end{enumerate}
Note that if $\bigcup_{\xi < \lambda} B_{\gamma,\xi} \subseteq \alpha$, $X \in [\alpha]^{<\aleph_{0}}$, then it makes sense to write $\{\xi : B_{\gamma,\xi} \otimes X \subseteq K_{\nu}\}$, since $K_{\nu} \cap [\alpha]^{2}$ has already been defined.    

Suppose we have constructed $K_{\nu}(\beta)$ for every $\nu < \lambda$ and $\beta < \alpha$.  Let $\langle (\nu_{\beta},\gamma_{\rho},X_{\rho}) : \rho < \lambda \rangle$ enumerate all triples $(\nu,\gamma,X)$ satisfying the hypothesis of $(*)$ for $\alpha$.  Apply Lemma \ref{matrix lemma} with $A = \alpha$, $I_{\rho} = \{\xi : B_{\gamma_{\rho},\xi} \otimes X_{\rho} \subseteq K_{\nu_{\rho}}\}$, and $E^{\xi}_{\rho} = B_{\gamma_{\rho},\xi}$ to obtain the disjoint sets $A_{\nu} := K_{\nu}(\alpha)$ for all $\nu < \lambda$.  Then for all $\nu < \lambda$, we have that if $\gamma < \alpha$, $\bigcup_{\xi < \lambda} B_{\gamma,\xi} \subseteq \alpha$, $X \in [\alpha]^{<\aleph_{0}}$, and $|\{\xi : B_{\gamma,\xi} \otimes X \subseteq K_{\nu}\}| = \lambda$, then $|\{\xi : B_{\gamma,\xi} \otimes X \subseteq K_{\nu} \text{ and } B_{\gamma,\xi} \subseteq K_{\nu}(\alpha)\}| = \lambda$.  Since the set $\{\xi : B_{\gamma,\xi} \otimes (X \cup \{\alpha\}) \subseteq K_{\nu}\}$ is equal to the set $\{\xi : B_{\gamma,\xi} \otimes X \subseteq K_{\nu} \text{ and } B_{\gamma,\xi} \subseteq K_{\nu}(\alpha)\}$, by the definition of $K_{\nu}(\alpha)$, this completes the proof of the claim. \qed

\textbf{Claim 2}:  If $\nu < \lambda$ and $\langle v_{\xi} : \xi < \lambda^{+} \rangle$ is a sequence of pairwise disjoint finite subsets of $\lambda^{+}$, then there are $\xi < \eta < \lambda$ so that $v_{\xi} \otimes v_{\eta} \subseteq K_{\nu}$.  

\emph{Proof of claim}:  There is an index $\gamma < \lambda^{+}$ such that $\overline{B}_{\gamma,\xi} = v_{\xi}$ for all $\xi < \lambda$.  By the regularity of $\lambda^{+}$, there is some $\beta < \lambda^{+}$ so that $\bigcup_{\xi < \lambda} v_{\xi} \subseteq \beta$ and we may further choose $\beta$ so that $\gamma < \beta$.  Since the sets $v_{\xi}$ are pairwise disjoint, there is some $\eta$ with $\lambda \leq \eta < \lambda^{+}$ so that $v_{\eta} \cap \beta = \emptyset$.  It follows that $\gamma < \alpha$ and $\bigcup_{\xi < \lambda}  B_{\gamma,\xi} \subseteq \alpha$ for all $\alpha \in v_{\eta}$.  List $v_{\eta} = \{\alpha_{0} < \ldots < \alpha_{m-1}\}$.  Applying the implication (A)$\implies$(B) of Claim 1 $m$ times, with $\alpha_{0},\ldots, \alpha_{m-1}$ playing the role of $\alpha$ and $\emptyset$, $\{\alpha_{0}\}$, \ldots, $\{\alpha_{0},\ldots, \alpha_{m-1}\}$ playing the role of $X$ in (A), we get that 
$$
|\{\xi < \lambda : B_{\gamma, \xi} \otimes v_{\eta} \subseteq K_{\nu}\} | = \lambda.
$$
In particular, there is some $\xi < \lambda \leq \eta$ so that $v_{\xi} \otimes v_{\eta} \subseteq K_{\nu}$.  \qed

Now to complete the proof, we must construct a coloring.  By replacing $K_{0}$ with $[\lambda^{+}]^{2} \setminus \left(\bigcup_{\nu > 0} K_{\nu}\right)$, we may assume that $\bigcup K_{\nu} = [\lambda^{+}]^{2}$.  We define a coloring $c : [\lambda^{+}]^{2} \to \lambda^{+}$ by $c(\{\alpha,\beta\}) = \nu$ if and only if $\{\alpha,\beta\} \in K_{\nu}$, for all $\nu < \lambda^{+}$, which is well-defined since the $K_{\nu}$ are pairwise disjoint with union $[\lambda^{+}]^{2}$.  Given any sequence $\langle v_{\xi} : \xi < \lambda^{+}\rangle$ of pairwise disjoint finite subsets of $\lambda^{+}$, we know by the regularity of $\lambda^{+}$ that there is a subsequence $\langle v_{\xi_{\rho}} : \rho < \lambda^{+} \rangle$ so that $\rho < \rho'$ implies $\max (v_{\xi_{\rho}}) < \min (v_{\xi_{\rho'}})$, so, replacing the given sequence by a subsequence, we may assume $\xi < \xi'$ implies $\max(v_{\xi}) < \min(v_{\xi'})$.  Given $\nu < \lambda^{+}$, we know, by Claim 2, there are $\xi < \eta< \lambda^{+}$ so that $v_{\xi} \otimes v_{\eta} \subseteq K_{\nu}$ or, in other words, $c(\{\alpha,\beta\}) = \nu$ for all $\alpha \in v_{\xi}$ and $\beta \in v_{\eta}$ which shows $c$ witnesses $\mathrm{Pr}_{1}(\lambda^{+},\lambda^{+},\lambda^{+},\aleph_{0})$.   
\end{proof}

\section{The main construction}

From strong colorings, we construct theories with \(\kappa_{\text{sct}}(T) + \kappa_{\text{inp}}(T) < \kappa_{\text{cdt}}(T)\).  For each regular uncountable cardinal \(\kappa\) and coloring \(f: [\kappa]^{2} \to 2\) we build a theory \(T^{*}_{\kappa,f}\) which comes equipped with a canonical cdt-pattern of height \(\kappa\), in which the consistency of two incomparable nodes, one on level \(\alpha\) and another on level \(\beta\), is determined by the value of the coloring \(f(\{\alpha, \beta\})\).  In the next section, we then analyze the possible inp- and sct-patterns that arise in models of \(T^{*}_{\kappa,f}\) and show that the combinatorial properties of the function $f$ are reflected in the values of the cardinal invariants $\kappa_{\mathrm{inp}}$ and $\kappa_{\mathrm{sct}}$.  

\subsection{Building a Theory}  Suppose \(\kappa\) is a regular uncountable cardinal.  We define a language \(L_{\kappa}  = \langle O,P_{\alpha}, f_{\alpha \beta}, p_{\alpha} : \alpha \leq \beta < \kappa\rangle\), where \(O\) and all the \(P_{\alpha}\) are unary predicates and the \(f_{\alpha \beta}\) and \(p_{\alpha}\) are unary functions.  Given a subset \(w \subseteq \kappa\), let \(L_{w} = \langle O, P_{\alpha}, f_{\alpha \beta},p_{\alpha} : \alpha \leq \beta, \alpha,\beta \in w\rangle\).  Given a function \(f: [\kappa]^{2} \to 2\), we define a universal theory \(T_{\kappa, f}\) with the following axiom schemas:
\begin{enumerate}
\item The predicates \(O\) and \((P_{\alpha})_{\alpha < \kappa}\) are pairwise disjoint; 
\item For all $\alpha < \kappa$, \(f_{\alpha\alpha}\) is the identity function, for all \(\alpha < \beta<\kappa\),
\[
(\forall x)\left[ (x \not\in P_{\beta} \to f_{\alpha \beta}(x) = x) \wedge (x \in P_{\beta} \to f_{\alpha \beta}(x) \in P_{\alpha})\right],
\]
and if \(\alpha < \beta < \gamma<\kappa\), then 
\[
(\forall x \in P_{\gamma})[f_{\alpha \gamma} (x) = (f_{\alpha \beta} \circ f_{\beta \gamma})(x)].
\] 
\item For all \(\alpha < \kappa\), 
\[
(\forall x)\left[(x \not\in O \to p_{\alpha}(x) = x) \wedge (p_{\alpha}(x) \neq x \to p_{\alpha}(x) \in P_{\alpha})\right].
\]
\item For all \(\alpha < \beta < \kappa\) satisfying \(f(\{\alpha , \beta\}) = 0\), we have the axiom 
$$
(\forall z \in O)[p_{\alpha}(z) \neq z \wedge p_{\beta}(z) \neq z \to p_{\alpha}(z) = (f_{\alpha \beta} \circ p_{\beta})(z)].
$$
\end{enumerate}

The \(O\) is for ``objects" and \(\bigcup P_{\alpha}\) is a tree of ``parameters'' where each \(P_{\alpha}\) names nodes of level \(\alpha\).  The functions \(f_{\alpha \beta}\) map elements of the tree at level \(\beta\) to their unique ancestor at level \(\alpha\).  So the tree partial order is coded in a highly non-uniform way, for each pair of levels.  The \(p_{\alpha}\)'s should be considered as partial functions on \(O\) which connect objects to elements of the tree: we will write $\text{dom}(p_{\alpha})$ for the set $\{x \in O : p_{\alpha}(x) \neq x\}$.  Axiom \((4)\) says, in essence, that if \(f(\{\alpha, \beta\}) = 0\), then the only way for an object in both $\text{dom}(p_{\alpha})$ and $\text{dom}(p_{\beta})$ to connect to a node on level \(\alpha\) and a node on level \(\beta\) is if these two nodes lie along a path in the tree.  

\begin{lem}
Define a class of finite structures 
\[
\mathbb{K}_{w} = \{ \text{ finite models of }T_{\kappa,f} \upharpoonright L_{w}\}.
\]
Then for finite \(w\), \(\mathbb{K}_{w}\) is a Fra\"iss\'e class and, moreover, it is uniformly locally finite.  
\end{lem}

\begin{proof}
The axioms for \(T_{\kappa,f}\) are universal so HP is clear.  JEP and AP are proved similarly, so we will give the argument for AP only.  Suppose \(A\) includes into \(B\) and \(C\) where \(A,B,C \in \mathbb{K}_{w}\) and \(B \cap C = A\).  Because all the symbols of the language are unary, \(B \cup C\) may be viewed as an \(L_{w}\)-structure by interpreting each predicate \(Q\) of \(L_{w}\) so that \(Q^{B \cup C} = Q^{B} \cup Q^{C}\) and similarly interpreting \(g^{B \cup C} = g^{B} \cup g^{C}\) for all the function symbols \(g \in L_{w}\).  It is easy to check that \(B \cup C\) is a model of \(T_{\kappa,f} \upharpoonright L_{w}\).  To see uniform local finiteness, just observe that a set of size \(n\) can generate a model of size at most \((|w|+1)n\) in virtue of the way that the functions are defined.  
\end{proof}

Hence, for each finite \(w \subset \kappa\), there is a countable ultrahomogeneous \(L_{w}\)-structure \(M_{w}\) with \(\text{Age}(M_{w}) = \mathbb{K}_{w}\).  Let \(T^{*}_{w} = \text{Th}(M_{w})\).  In the following lemmas, we will establish the properties needed to apply Lemma \ref{KPTredux} in order to show the $T^{*}_{w}$ cohere.  

\begin{lem} \label{first condition}
Suppose $w \subseteq v$ are finite subsets of $\kappa$ and $A \in \mathbb{K}_{w}$.  Then there is an $L_{v}$-structure $D \in \mathbb{K}_{v}$ such that $A \subseteq D \upharpoonright L_{w}$.	
\end{lem}

\begin{proof}
We may enumerate $w$ in increasing order as $w = \{\alpha_{0} < \alpha_{1} < \ldots < \alpha_{n-1}\}$.  By induction, it suffices to consider the case when $v = w \cup \{\gamma\}$ for some $\gamma \in \kappa \setminus w$.  We consider two cases:

\textbf{Case 1}:  $\alpha_{n-1} < \gamma$ or $w = \emptyset$.  

In this case, the new symbols in $L_{v}$ not in $L_{w}$ consist of the predicate $P_{\gamma}$, the function $p_{\gamma}$, and the functions $f_{\alpha_{j}\gamma}$ for $j < n$ and $f_{\gamma \gamma}$.  We define the underlying set of $D$ to be $A$, and give the symbols of $L_{w}$ their interpretation in $A$.  Then we interpret $P_{\gamma}^{D} = \emptyset$, and interpret $p_{\gamma}^{D}$, $f_{\alpha_{j}\gamma}^{D}$ for $j < n$, and $f^{D}_{\gamma \gamma}$ to be the identity function on $D$.  Clearly $A = D \upharpoonright L_{w}$ and it is easy to check $D \in \mathbb{K}_{v}$.  

\textbf{Case 2}:  $\gamma < \alpha_{n-1}$.

Let $i$ be least such that $\gamma < \alpha_{i}$.  We define the underlying set of $D$ to be $A \cup \{*_{d} : d \in P_{\alpha_{i}}^{A}\}$, where the $*_{d}$ denote new formal elements.  We interpret all the predicates of $L_{w}$ on $D$ to have the same interpretation on $A$, and we interpret each function of $L_{w}$ to be the identity on $\{*_{d}: d \in P^{A}_{\alpha_{i}}\}$ and, when restricted to $A$, to have the same interpretation as in $A$. The new symbols in $L_{v}$ not in $L_{w}$ are:  the predicate $P_{\gamma}$, the function $p_{\gamma}$, and the functions $f_{\alpha_{j}\gamma}$ for $j < i$, the function $f_{\gamma \gamma}$, and the functions $f_{\gamma \alpha_{j}}$ for $i \leq j < n$.  We remark that it is possible that $i = 0$, in which case there are no such $j < i$ so our conditions on $f_{\alpha_{j}\gamma}$ below say nothing.  We interpret $P_{\gamma}^{D} = \{*_{d} : d \in P^{A}_{\alpha_{i}}\}$	and $p_{\gamma}^{D}$ as the identity function on $D$.  Informally speaking, we will interpret the remaining functions so that $*_{d}$ becomes the ancestor of $d$ at level $\gamma$.  More precisely, for $j < i$, we set $f^{D}_{\alpha_{j} \gamma}(*_{d}) = f^{A}_{\alpha_{j}\alpha_{i}}(d)$ and to be the identity on the complement of $\{*_{d} : d \in P^{A}_{\alpha_{i}}\}$.  Likewise, if $i \leq j < n$ and $e \in P^{D}_{\alpha_{j}}$, we set $f^{D}_{\gamma \alpha_{j}}(e) = *_{f_{\alpha_{i}\alpha_{j}}^{A}(e)}$ and we define $f^{D}_{\gamma \alpha_{j}}$ to be the identity on the complement of $P^{D}_{\alpha_{j}}$.  Finally, we set $f^{D}_{\gamma \gamma} = \text{id}_{D}$, which completes the definition of the $L_{v}$-structure $D$.  

Now we check that $D \in \mathbb{K}_{v}$.  By construction and the fact that $\mathbb{A} \in \mathbb{K}_{w}$, all the axioms are clear except, in order to establish (2), we must check that if $\beta < \beta' < \beta''$ are from $v$, then for all $x \in P^{D}_{\beta''}$, $(f^{D}_{\beta \beta'} \circ f^{D}_{\beta'\beta''})(x) = f^{D}_{\beta \beta''}(x)$.  We may assume $\gamma \in \{\beta,\beta',\beta''\}$.  If $\gamma = \beta''$, then every element of $P_{\gamma}^{D}$ is of the form $*_{d}$ for some $d \in P^{A}_{\alpha_{i}}$ and we have 
\begin{eqnarray*}
(f_{\beta\beta'}^{D} \circ f^{D}_{\beta' \gamma})(*_{d}) &=& (f^{D}_{\beta \beta'} \circ f^{D}_{\beta' \alpha_{i}})(d)\\
	&=& f^{D}_{\beta \alpha_{i}}(d) \\
	&=& f^{D}_{\beta \gamma}(*_{d}),
\end{eqnarray*} 
by the definition of $f^{D}_{\alpha_{j}\gamma}$ for $j < i$ and the fact that $D$ extends $A$, which satisfies axiom (2).  Similarly, if $\gamma = \beta'$ and $x \in P^{D}_{\beta''}$, we have 
\begin{eqnarray*}
(f^{D}_{\beta \gamma} \circ f^{D}_{\gamma \beta''})(x)&=& f^{D}_{\beta \gamma}(*_{f^{D}_{\alpha_{i}\beta''}(x)})\\
&=& f^{D}_{\beta \alpha_{i}}(f^{D}_{\alpha_{i}\beta''}(x)) \\
&=& f^{D}_{\beta \beta''}(x).	
\end{eqnarray*}
Finally, if $\beta = \gamma$ and $x \in P^{D}_{\beta''}$, we have 
\begin{eqnarray*}
f^{D}_{\gamma \beta'}(f^{D}_{\beta' \beta''}(x)) &=& *_{f^{D}_{\alpha_{i}\beta'}(f^{D}_{\beta' \beta''}(x))} \\
&=& *_{f^{D}_{\alpha_{i}\beta''}(x)} \\
&=& f^{D}_{\gamma \beta''}(x),
\end{eqnarray*}
which verifies that (2) holds of $D$ and therefore $D \in \mathbb{K}_{v}$.  
\end{proof}

\begin{lem} \label{second condition}
Suppose $w \subseteq v$ are finite subsets of $\kappa$, $A,B \in \mathbb{K}_{w}$, and $\pi: A \to B$ is an $L_{w}$-embedding.  Then given any $C \in \mathbb{K}_{v}$ with $C = \langle A \rangle^{C}_{L_{v}}$, there is $D \in \mathbb{K}_{v}$ and an $L_{v}$-embedding $\tilde{\pi}: C \to D$ extending $\pi$.  
\end{lem}

\begin{proof}
As in the proof of Lemma \ref{first condition}, we will list $w$ in increasing order as $w = \{\alpha_{0} < \alpha_{1} < \ldots < \alpha_{n-1}\}$ and assume that $v = w \cup \{\gamma\}$ for some $\gamma \in \kappa \setminus w$.  We suppose we are given $A,B,C$, and $\pi$ as in the statement and we will construct $D$ and $\tilde{\pi}$.  We may assume $B \cap C = \emptyset$.  Note that the condition that $C = \langle A \rangle^{C}_{L_{v}}$ entails that the only elements of $C \setminus A$ are contained in $P_{\gamma}^{C}$ and similarly for $B$ and $D$.    

\textbf{Case 1}:  $\alpha_{n-1} < \gamma$ or $w = \emptyset$.

We define the underlying set of $D$ to be $B \cup P^{C}_{\gamma}$ and we define $\tilde{\pi}: C \to D$ so that $\tilde{\pi}\upharpoonright A = \pi$ and $\tilde{\pi} \upharpoonright P^{C}_{\gamma} = \text{id}_{P^{C}_{\gamma}}$.  Interpret the predicates of $L_{w}$ on $D$ so that they agree with their interpretation on $B$ and interpret the functions of $L_{w}$ on $D$ so that they are the identity on $P^{C}_{\gamma}$ and so that, when restricted to $B$, they agree with their interpretation on $B$.  This will ensure that $D \upharpoonright L_{w}$ is an extension of $B$.   Finally, interpret $P_{\gamma}$ so that $P_{\gamma}^{D} = P^{C}_{\gamma}$ and define $f^{D}_{\gamma \gamma} = \text{id}_{D}$.  Then for each $j < n$, we interpret $f_{\alpha_{j}\gamma}$ on $D$ so that, if $c \in P^{C}_{\gamma}$, then $f^{D}_{\alpha_{j} \gamma}(c) = \pi(f^{C}_{\alpha_{j}\gamma}(c))$, and if $c \in D \setminus P^{C}_{\gamma}$, then $f^{D}_{\alpha_{j}\gamma}(c) = c$.  Note that $\tilde{\pi}(f^{C}_{\alpha_{j}\gamma}(c)) = f^{D}_{\alpha_{j}\gamma}(\tilde{\pi}(c))$ for all $c \in C$.  Finally, interpret $p_{\gamma}$ so that, if $d = \pi(c) \in \pi(O^{C}) \subseteq O^{D}$ and $p^{C}_{\gamma}(c) \neq c$, then $p^{D}_{\gamma}(d) = p^{C}_{\gamma}(c)$, and otherwise $p^{D}_{\gamma}(d) = d$.  It is clear from the definitions that $\tilde{\pi}(p_{\gamma}^{C}(c)) = p_{\gamma}^{D}(\tilde{\pi}(c))$ for all $c \in C$, so $\tilde{\pi}$ is an $L_{v}$-embedding.  We are left with showing that $D \in \mathbb{K}_{v}$.  Axioms (1) and (3) are clear from the construction and to check (2), we just need to establish that if $\beta < \beta'$ are from $v$ and $c \in P^{C}_{\gamma}$, then $(f^{D}_{\beta \beta'} \circ f^{D}_{\beta' \gamma})(c) = f^{D}_{\beta \gamma}(c)$.  For this, we unravel the definitions and make use of the fact that (2) is true in $C$:
\begin{eqnarray*}
f^{D}_{\beta \beta'}(f^{D}_{\beta'\gamma}(c)) &=& f^{D}_{\beta \beta'}(\pi(f^{C}_{\beta'\gamma}(c)) \\
&=& \pi (f^{C}_{\beta \beta'}(f^{C}_{\beta' \gamma}(c))\\
&=& \pi(f^{C}_{\beta \gamma}(c)) \\
&=& f^{D}_{\beta \gamma}(c),	
\end{eqnarray*}
which verifies (2).  Likewise, to show that (4) holds, we note that if $f(\{\beta,\gamma\}) = 0$, $p^{D}_{\gamma}(d) \neq d$, and $p^{D}_{\beta}(d) \neq d$ for some $\beta \in v$ then, by definition of $p^{D}_{\gamma}$, $d = \tilde{\pi}(c)$ for some $c \in O^{C}$ so $p^{C}_{\beta}(c) = (f_{\beta \gamma}^{C} \circ p^{C}_{\gamma})(c)$ so $p^{D}_{\beta}(d) = (f_{\beta \gamma}^{D} \circ p_{\gamma}^{D})(d)$ as $\tilde{\pi}$ is an embedding, which shows (4) and thus $D \in \mathbb{K}_{v}$.

\textbf{Case 2}:  $\gamma < \alpha_{n-1}$.

Let $i$ be least such that $\gamma < \alpha_{i}$.  The underlying set of $D$ will be $B \cup P_{\gamma}^{C} \cup \{*_{d} : d \in P^{B}_{\alpha_{i}} \setminus \pi(P^{A}_{\alpha_{i}})\}$, where each $*_{d}$ denotes a new formal element and we will define $\tilde{\pi}: C \to D$ to be $\pi \cup \text{id}_{P_{\gamma}^{C}}$.  As in the previous case, we interpret the predicates of $L_{w}$ on $D$ so that they agree with their interpretation on $B$ and interpret the functions of $L_{w}$ on $D$ so that they are the identity on $P^{C}_{\gamma} \cup \{*_{d} : d \in P^{B}_{\alpha_{i}} \setminus \pi(P^{A}_{\alpha_{i}})\} $ and so that, when restricted to $B$, they agree with their interpretation on $B$. We will interpret $P_{\gamma}$ so that 
$$
P^{D}_{\gamma} = P^{C}_{\gamma} \cup \{*_{d} : d \in P^{B}_{\alpha_{i}} \setminus \pi(P^{A}_{\alpha_{i}})\}
$$
The map $\pi$ will dictate how we have to define the ancestors and descendants at level $\gamma$ of the elements in the image of $\pi$, and, for those elements not in the image of $\pi$, we define the interpretations so that $*_{d}$ will be the ancestor at level $\gamma$ of $d \in P^{B}_{\alpha_{i}} \setminus \pi(P^{A}_{\alpha_{i}})$, as in the previous lemma.  For $j < i$, we define $f^{D}_{\alpha_{j}\gamma}$ so that, if $c \in P^{C}_{\gamma}$, $f^{D}_{\alpha_{j}\gamma}(c) = \pi(f^{C}_{\alpha_{j}\gamma}(c))$, and if $d \in P^{B}_{\alpha_{i}} \setminus \pi(P^{A}_{\alpha_{i}})$, then $f^{D}_{\alpha_{j}\gamma}(*_{d}) = f^{B}_{\alpha_{j}\alpha_{i}}(d)$.  This defines $f^{D}_{\alpha_{j}\gamma}$ on $P^{D}_{\gamma}$ and we define $f^{D}_{\alpha_{j}\gamma}$ to be the identity on the complement of $P^{D}_{\gamma}$ in $D$.  Next, we define $f^{D}_{\gamma \alpha_{i}}$ as follows:  if $d = \pi(c) \in \pi(P^{C}_{\alpha_{i}}) \subseteq P^{B}_{\alpha_{i}}$, we put $f_{\gamma \alpha_{i}}^{D}(d) = f^{C}_{\gamma \alpha_{j}}(c)$, and if $e \in P^{B}_{\alpha_{i}} \setminus \pi(P^{C}_{\alpha_{i}})$, then we set $f^{D}_{\gamma \alpha_{i}}(e) = *_{e}$.  This defines $f^{D}_{\gamma \alpha_{i}}$ on $P^{D}_{\alpha_{i}}$ and we define $f^{D}_{\gamma \alpha_{i}}$ to be the identity on the complement of $P^{D}_{\alpha_{i}}$ in $D$.  For $j > i$, we put $f^{D}_{\gamma \alpha_{j}} = f^{D}_{\gamma \alpha_{i}} \circ f^{D}_{\alpha_{i} \alpha_{j}}$.  Then we define $f_{\gamma \gamma} = \text{id}_{D}$.  Lastly, we define $p^{D}_{\gamma}$ to be the identity on all elements in the complement of $\pi(O^{A})$ and if $d = \pi(c)$, we put $p_{\gamma}^{D}(d) = d$ if $p^{C}_{\gamma}(c) = c$ and we put $p_{\gamma}^{D}(d) = p^{C}_{\gamma}(c)$ if $p^{C}_{\gamma}(c) \neq c$.  This completes the construction.  

It follows from the definitions that $\tilde{\pi}$ is an $L_{v}$-embedding, so we must check $D \in \mathbb{K}_{v}$.  Axioms (1) and (3) are clear from the construction.  To show (2), we note that if $\beta < \beta' < \beta''$ and $c \in P^{D}_{\beta''}$, then either $c$ is in the image of $\tilde{\pi}$, in which case it is easy to check that $(f^{D}_{\beta \beta'} \circ f^{D}_{\beta' \beta''})(c) = f^{D}_{\beta \beta''}(c)$ using that (2) is satisfied in $C$ and $\tilde{\pi}$ is an embedding, or $c$ is not in the image of $\pi$, in which case the verification of (2) is identical to the verification of (2) in Case 2 of Lemma \ref{first condition}.  The argument for (4) is identical to the argument for (4) in Case 1.  We conclude that $D \in \mathbb{K}_{v}$, completing the proof.
\end{proof}

\begin{cor}\label{dirlim}
Suppose \(w \subseteq v \subseteq \kappa\) and \(v,w\) are both finite.  Then \(T^{*}_{w} \subseteq T^{*}_{v}\).  
\end{cor}

\begin{proof}
We will show $\mathrm{Flim}(\mathbb{K}_{v})\upharpoonright L_{w} = \mathrm{Flim}(\mathbb{K}_{w})$ by applying Lemma \ref{KPTredux}. Condition (1) in the Lemma is proved in Lemma \ref{first condition} and Condition (2) is proved in Lemma \ref{second condition}.  
\end{proof}

Using Corollary \ref{dirlim}, we may define the theory \(T^{*}_{\kappa,f}\) as the union of the $T^{*}_{w}$ for all finite $w \subset \kappa$ and the resulting theory is consistent.  Because each $T^{*}_{w}$ is complete and eliminates quantifiers, it follows that $T^{*}_{\kappa,f}$ is a complete theory extending $T_{\kappa,f}$ which eliminates quantifiers.  

%

The following lemmas will be useful in analyzing the possible formulas that could appear in the various patterns under consideration.  Recall that, for all $\alpha < \kappa$, we write $\text{dom}(p_{\alpha})$ for the definable set $\{x \in O : p_{\alpha}(x) \neq x\}$, or equivalently $\{x \in O : p_{\alpha}(x) \in P_{\alpha}\}$.

\begin{lem} \label{nice functions on P}
Suppose $w \subseteq \kappa$ is a finite set containing $\beta$ and $\varphi(x)$ is an $L_{w}$-formula with $\varphi(x) \vdash x \in P_{\beta}$.  Then for any $L_{w}$-term $t(x)$, there is $\alpha \leq \beta$ in $w$ such that $\varphi(x) \vdash t(x) =f_{\alpha \beta}(x)$.	
\end{lem}

\begin{proof}
The proof is by induction on terms.  The conclusion holds for the term $x$ since $(\forall x)[f_{\beta \beta}(x) =x]$ is an axiom of $T_{\kappa,f}$.  Now suppose $t(x)$ is a term such that $\varphi(x) \vdash t(x) = f_{\alpha \beta}(x)$ for some $\alpha \leq \beta$ from $w$.  Then because $\varphi(x) \vdash x \in P_{\beta}$, $\varphi(x) \vdash t(x) \in P_{\alpha}$.  It follows that for any $\delta \leq \gamma$ from $w$, $\varphi(x) \vdash p_{\gamma}(t(x)) = t(x)$ and $\varphi(x) \vdash f_{\delta \gamma}(t(x)) = t(x)$ when $\gamma \neq \alpha$.  Additionally, if $\delta \leq \alpha$ is from $w$, then $\varphi(x) \vdash f_{\delta \alpha}(t(x)) = (f_{\delta \alpha} \circ f_{\alpha \beta})(x) = f_{\delta \beta}(x)$, which is of the desired form, completing the induction.  	
\end{proof}

\begin{lem} \label{nice functions on O}
Suppose $w \subseteq \kappa$ is finite and $\varphi(x)$ is a complete $L_{w}$-formula with $\varphi(x) \vdash x \in O$.  Then for any term $t(x)$ of $L_{w}$, we have one of the following:
\begin{enumerate}
\item $\varphi(x) \vdash t(x) =x$.
\item $\varphi(x) \vdash t(x) = (f_{\alpha \beta} \circ p_{\beta})(x)$ for some $\alpha \leq \beta$ from $w$. 
\end{enumerate}
\end{lem}

\begin{proof}
The proof is by induction on terms.  Clearly the conclusion holds for the term $t(x) = x$.  Now suppose we have established the conclusion for the term $t(x)$.  We must prove that it also holds for the terms $p_{\gamma}(t(x))$ and $f_{\delta \gamma}(t(x))$ for $\delta \leq \gamma$ from $w$.  If $\varphi(x) \vdash t(x) = x$, then $\varphi(x) \vdash p_{\gamma}(t(x)) = (f_{\gamma \gamma} \circ p_{\gamma})(x)$, which falls under case (2), and $\varphi(x) \vdash f_{\delta \gamma}(t(x)) = x$, since $\varphi(x) \vdash t(x) \in O$ which is under case (1).  
%

Now suppose $\varphi(x) \vdash t(x) = (f_{\alpha \beta} \circ p_{\beta})(x)$.  Since we already handled terms falling under case (1), we may, by completeness of $\varphi$, assume $\varphi(x) \vdash x \in \text{dom}(p_{\beta})$ and hence $\varphi(x) \vdash t(x) \in P_{\alpha}$. It follows that $\varphi(x) \vdash p_{\gamma}(t(x)) = t(x)$ and $\varphi(x) \vdash f_{\delta\gamma}(t(x)) = t(x)$ when $\gamma \neq \alpha$, which remain under case (2).  Finally, we have $\varphi(x) \vdash f_{\delta \alpha}(t(x)) = (f_{\delta\alpha} \circ f_{\alpha \beta} \circ p_{\beta})(x) = (f_{\delta \beta} \circ p_{\beta})(x)$, which also remains under case (2), completing the induction.  
\end{proof}

\section{Analysis of the invariants}

In this section, we analyze the possible values of the cardinal invariants under consideration in $T^{*}_{\kappa,f}$ for a coloring $f:[\kappa]^{2} \to 2$.  In the first subsection, we show that any $\mathrm{inp}$- and $\mathrm{sct}$-pattern of height $\kappa$ in $T^{*}_{\kappa, f}$ gives rise to one of a particularly uniform and controlled form, which we call \emph{rectified}. In the second subsection, we show $\kappa_{cdt}(T^{*}_{\kappa,f}) = \kappa^{+}$, independent of the choice of $f$.  Then, making heavy use of rectification, we show in the next two subsections that if $\kappa_{\text{sct}}(T^{*}_{\kappa,f})$ or $\kappa_{\text{inp}}(T^{*}_{\kappa,f})$ are equal to $\kappa^{+}$, then this has combinatorial consequences for the coloring $f$.  More precisely, we show in the third subsection that if there is an inp-pattern of height \(\kappa\), we can conclude that \(f\) has a homogeneous set of size \(\kappa\).  In the case that there is an sct-pattern of height \(\kappa\), we cannot quite get a homogeneous set, but one nearly so:  we prove in this case that there is precisely the kind of homogeneity which a strong coloring witnessing $\text{Pr}_{1}(\kappa, \kappa, 2, \aleph_{0})$ explicitly prohibits.  The theory associated to such a coloring, then, gives the desired counterexample. 

For the entirety of this section, we will fix $\kappa$ a regular uncountable cardinal, a coloring $f:[\kappa]^{2} \to 2$, and a monster model \(\mathbb{M} \models T^{*}_{\kappa,f}\).  

\subsection{Rectification}

Recall that, given a set $X$, a family of subsets $\mathcal{B} \subseteq \mathcal{P}(X)$ is called a \emph{$\Delta$-system} (of subsets of $X$) if there is some $r \subseteq X$ such that for all distinct $x,y \in \mathcal{B}$, $x \cap y = r$.  Given a $\Delta$-system, the common intersection of any two distinct sets is called the \emph{root} of the $\Delta$-system.  The following fact gives a condition under which $\Delta$-systems may be shown to exist:

\begin{fact} \cite[Lemma III.2.6]{kunen2014set} \label{delta-system lemma}
Suppose that 	$\lambda$ is a regular uncountable cardinal and $\mathcal{A}$ is a family of finite subsets of $\lambda$ with $|\mathcal{A}| = \lambda$.  Then there is $\mathcal{B} \subseteq \mathcal{A}$ with $|\mathcal{B}| = \lambda$ and which forms a $\Delta$-system.  
\end{fact}

We note that the definitions below are specific to $T^{*}_{\kappa,f}$.  Recall that, given a subset \(w \subseteq \kappa\), we define \(L_{w} = \langle O, P_{\alpha}, f_{\alpha \beta},p_{\alpha} : \alpha \leq \beta, \alpha,\beta \in w\rangle\).

\begin{defn}
Given $X \in \{\mathrm{inp},\mathrm{sct}\}$, we define a \emph{rectified $X$-pattern as follows}:
\begin{enumerate}
\item A \emph{rectified $\mathrm{sct}$-pattern of height $\kappa$} is a triple $(\overline{\varphi},(a_{\eta})_{\eta \in \omega^{<\kappa}},\overline{w})$ satisfying the following:
\begin{enumerate}
\item $(a_{\eta})_{\eta \in \omega^<{\kappa}}$ is an $s$-indiscernible tree of parameters.  
\item $\overline{\varphi}$ is a sequence of formulas $(\varphi_{\alpha}(x;y_{\alpha}) : \alpha < \kappa)$ which, together with the parameters $(a_{\eta})_{\eta \in \omega^{<\kappa}}$ forms an $\mathrm{sct}$-pattern of height $\kappa$.  	
\item $\overline{w} = (w_{\alpha})_{\alpha < \kappa}$ is a $\Delta$-system of finite subsets of $\kappa$ with root $r$ such that every $w_{\alpha}$ has the same cardinality, $\max r < \min(w_{\alpha} \setminus r)$ for all $\alpha < \kappa$, and $\max(w_{\alpha} \setminus r) < \min(w_{\alpha'} \setminus r)$ for all $\alpha < \alpha' < \kappa$.  
\item For all $\alpha < \kappa$, the formula $\varphi_{\alpha}(x;y_{\alpha})$ is in the language $L_{w_{\alpha}}$ and isolates a complete $L_{w_{\alpha}}$-type over $\emptyset$ in the variables $xy_{\alpha}$.  Additionally, for all $\alpha < \kappa$ and $\eta \in \omega^{\alpha}$, the tuple $a_{\eta}$ enumerates an $L_{w_{\alpha}}$-substructure of $\mathbb{M}$.  
\end{enumerate}
\item We define a \emph{rectified $\mathrm{inp}$-pattern of height $\kappa$} to be a quadruple $(\overline{\varphi},\overline{k},(a_{\alpha,i})_{\alpha < \kappa, i < \omega},\overline{w})$ satisfying the following:
\begin{enumerate}
\item $(a_{\alpha ,i})_{\alpha < \kappa, i < \omega}$ is a mutually indiscernible array of parameters.  
\item $\overline{\varphi}$ is a sequence of formulas $(\varphi_{\alpha}(x;y_{\alpha}) : \alpha < \kappa)$ and $\overline{k} = (k_{\alpha})_{\alpha < \kappa}$ is a sequence of natural numbers which, together with the parameters $(a_{\alpha,i})_{\alpha < \kappa, i < \omega}$ form an $\mathrm{inp}$-pattern of height $\kappa$. 
\item $\overline{w} = (w_{\alpha})_{\alpha < \kappa}$ is a $\Delta$-system of finite subsets of $\kappa$ with root $r$ such that every $w_{\alpha}$ has the same cardinality, $\max r < \min(w_{\alpha} \setminus r)$ for all $\alpha < \kappa$, and $\max(w_{\alpha} \setminus r) < \min(w_{\alpha'} \setminus r)$ for all $\alpha < \alpha' < \kappa$.  
\item For all $\alpha < \kappa$, the formula $\varphi_{\alpha}(x;y_{\alpha})$ is in the language $L_{w_{\alpha}}$ and isolates a complete $L_{w_{\alpha}}$-type over $\emptyset$ in the variables $xy_{\alpha}$.  Additionally, for all $\alpha < \kappa$ and $i < \omega$, the tuple $a_{\alpha,i}$ enumerates an $L_{w_{\alpha}}$-substructure of $\mathbb{M}$.  
\end{enumerate}
\item We will refer to $\overline{w}$ in the above definitions as the \emph{associated $\Delta$-system} of the rectified $X$-pattern.  We will consistently denote the root $r = \{\zeta_{i} : i < n\}$ and the sets $v_{\alpha} = w_{\alpha} \setminus r = \{\beta_{\alpha, i} : i < m \}$, where the enumerations are increasing.  
\end{enumerate}
\end{defn}

\begin{lem}\label{tending}
Given \(X \in \{\text{inp},\text{sct}\}\), if there is an \(X\)-pattern of height \(\kappa\) in \(T\), there is a rectified one.  
\end{lem}

\begin{proof}
Given an \(X\)-pattern with the sequence of formulas \(\overline{\varphi} = (\varphi_{\alpha}(x;y_{\alpha}): \alpha < \kappa)\) one can choose some finite \(w_{\alpha} \subset \kappa\) such that \(\varphi_{\alpha}(x;y_{\alpha})\) is in the language \(L_{w_{\alpha}}\).  Apply the \(\Delta\)-system lemma, Fact \ref{delta-system lemma}, to the collection \((w_{\alpha} : \alpha < \kappa)\) to find some \(I \subseteq \kappa\) with $|I| = \kappa$ such that \(\overline{w} = (w_{\alpha} : \alpha \in I)\) forms a \(\Delta\)-system with root $r$.  By the pigeonhole principle, using that $\kappa$ is uncountable, and the regularity of $\kappa$, we may assume \(|w_{\alpha}| = m\) for all \(\alpha< \kappa\), \(\max r < \min (w_{\alpha} \setminus r)\) for all $\alpha < \kappa$, and if \(\alpha < \alpha'\), \(\max (w_{\alpha} \setminus r) < \min (w_{\alpha'} \setminus r)\).  By renaming, we may assume \(I = \kappa\).

If \(X = \text{inp}\), we may take the parameters witnessing that \((\overline{\varphi},\overline{k},(a_{\alpha,i})_{\alpha <\kappa,i < \omega})\) is an inp-pattern to be a mutually indiscernible array by Lemma \ref{witness}(1).  Moreover, mutual indiscernibility is clearly preserved after replacing each \(a_{\alpha,i}\) by a tuple enumerating the \(L_{w_{\alpha}}\)-substructure generated by $a_{\alpha,i}$ and, by \(\aleph_{0}\)-categoricity of \(T^{*}_{w_{\alpha}}\), this structure is finite.  Let \(b \models \{\varphi_{\alpha}(x;a_{\alpha,0}) : \alpha < \kappa\}\).  Using again the \(\aleph_{0}\)-categoricity of \(T^{*}_{w_{\alpha}}\), replace \(\varphi_{\alpha}(x;y_{\alpha})\) by an \(L_{w_{\alpha}}\)-formula \(\varphi'_{\alpha}(x;y_{\alpha})\) such that \(\varphi_{\alpha}'(x,y_{\alpha})\), viewed as an unpartitioned formula in the variables $xy_{\alpha}$, isolates the type \(\text{tp}_{L_{w_{\alpha}}}(ba_{\alpha,0}/\emptyset)\). By mutual indiscernibility, if \(g: \kappa \to \omega\) is a function, there is \(\sigma \in \text{Aut}(\mathbb{M})\) such that \(\sigma(a_{\alpha,0}) = a_{\alpha, g(\alpha)}\) for all \(\alpha < \kappa\).  Then \(\sigma(b) \models \{\varphi'_{\alpha}(x;a_{\alpha,g(\alpha)}) : \alpha < \kappa\}\) so paths are consistent.  The row-wise inconsistency is clear so if we set \(\overline{\varphi}' =(\varphi'_{\alpha}(x;y_{\alpha}) : \alpha < \kappa)\), we see $(\overline{\varphi}',\overline{k},(a_{\eta,i})_{\alpha < \kappa, i < \omega},\overline{w})$ forms a rectified inp-pattern of height $\kappa$.

If \(X = \text{sct}\), we argue similarly.  We may take the witnessing parameters $(a_{\eta})_{\eta \in \omega^{<\kappa}}$ to be \(s\)-indiscernible, by Lemma \ref{witness}(2).  Likewise, \(s\)-indiscernibility is preserved by replacing each \(a_{\eta}\) by its closure under the functions of \(L_{w_{l(\eta)}}\) and this closure is finite.  Let \(b \models \{\varphi_{\alpha}(x;a_{0^{\alpha}}) : \alpha < \kappa\}\) and replace $\overline{\varphi}$ by $\overline{\varphi}'$ where \(\varphi'_{\alpha}(x;y_{\alpha})\) is an \(L_{w_{\alpha}}\)-formula which, viewed as an unpartitioned formula in the variables $xy_{\alpha}$, isolates \(\mathrm{tp}_{L_{w_{\alpha}}}(ba_{0^{\alpha}}/\emptyset)\).  For all \(\eta \in \omega^{\kappa}\), there is a \(\sigma \in \text{Aut}(\mathbb{M})\) such that \(\sigma(a_{0^{\alpha}}) = a_{\eta | \alpha}\).  Then \(\sigma(b) \models \{\varphi'_{\alpha}(x;a_{\eta | \alpha}) : \alpha < \kappa\}\) so paths are consistent.  Incomparable nodes remain inconsistent, so \((\overline{\varphi}',(a_{\eta})_{\eta \in \omega^{<\kappa}},\overline{w})\) forms a rectified sct-pattern.  
\end{proof}

\begin{rem} \label{same number of variables}
As the replacement of $(\varphi(x;y_{\alpha}) : \alpha < \kappa)$ with a sequence of complete formulas $(\varphi'_{\alpha}(x;y_{\alpha}) : \alpha < \kappa)$ does not change the free variables $x$, if $T$ has an inp- or sct-pattern in $k$ free variables of height $\kappa$, Lemma \ref{tending} produces a rectified inp- or sct-pattern of height $\kappa$ in the same number of free variables.  
\end{rem}

\subsection{Computing \(\kappa_{cdt}\)}

\begin{lem}\label{stable}
The theory \(T_{\kappa,f}^{*}\) is stable.  
\end{lem}

\begin{proof}
Since stability is local, it suffices to show \(T^{*}_{w}\) is stable for all finite \(w \subset \kappa\).  Let \(M \models T_{w}\) be a countable model.  We will count 1-types in \(T^{*}_{w}\) over \(M\) explicitly using quantifier elimination.  Pick some \(p(x) \in S^{1}_{L_{w}}(M)\).  If \(x = m\) is a formula in \(p\) for some \(m \in M\) then this formula obviously isolates \(p\) so there are countably many such possibilities.  So assume \(x \neq m\) is in \(p\) for all \(m \in M\).  

Now we break into cases based upon the predicate contained in $p$.  If \(x \not\in O \wedge \bigwedge_{\alpha \in w} x \not\in P_{\alpha}\) is a formula in \(p\), then \(p\) is completely determined, so there is a unique type in this case.  If \(x \in O\) is a formula in \(p\), then, by quantifier-elimination and Lemma \ref{nice functions on O}, the type is determined after deciding the truth value of $p_{\alpha}(x)= x$ and \((f_{\beta \alpha} \circ p_{\alpha})(x) = m\) for all \(\beta \leq \alpha \in w\) and \(m \in P_{\beta}(M)\).  As \((f_{\beta \alpha} \circ p_{\alpha})(x)\) can be equal to at most 1 element of \(P_{\beta}(M)\) and $w$ is finite, there are countably many possibilties for this case.  Finally, if \(x \in P_{\beta}\) is a formula in \(p\), then, by quantifier-elimination and Lemma \ref{nice functions on P}, the type is determined after deciding the truth value of \(f_{\gamma \beta}(x) = m\) for \(m \in P_{\gamma}(M)\) for all \(\gamma < \beta < \alpha\) from \(w\).  Here again there are only countably many possibilities, by the finiteness of $w$.  Since this covers all possible types, we've shown that \(S^{1}_{L_{w}}(M)\) is countable, so \(T^{*}_{w}\) is stable (in fact, as $M$ is an arbitrary countable model, $\omega$-stable) which implies that $T^{*}_{\kappa,f}$ is stable.  
\end{proof}

\begin{prop} \label{cdtcomputation}
\(\kappa_{\text{cdt}}(T^{*}_{\kappa,f}) = \kappa^{+}\).  
\end{prop}

\begin{proof}
First, we will show $\kappa_{\text{cdt}}(T^{*}_{\kappa,f}) \geq \kappa^{+}$.  We will construct a cdt-pattern of height \(\kappa\).  By recursion on \(\alpha < \kappa\), we will construct a tree of tuples \((a_{\eta})_{\eta \in \omega^{<\kappa}}\) so that \(l(\eta) = \beta\) implies \(a_{\eta} \in P_{\beta}\) and if \(\eta \unlhd \nu\) with \(l(\eta) = \beta\) and \(l(\nu) = \gamma\), then \(f_{\beta \gamma}(a_{\nu}) = a_{\eta}\).  For \(\alpha = 0\), choose an arbitrary \(a \in P_{0}\) and let \(a_{\emptyset} = a\).  Now suppose given \((a_{\eta})_{\eta \in \omega^{\leq \alpha}}\).  For each \(\eta \in \omega^{\alpha}\), choose a set \(\{b_{i} : i < \omega\} \subseteq f^{-1}_{\alpha \alpha+1}(a_{\eta})\) with the $b_{i}$ pairwise distinct.  Define \(a_{\eta \frown \langle i \rangle} = b_{i}\).  This gives us \((a_{\eta})_{\eta \in \omega^{\leq \alpha+1}}\) with the desired properties.  Now suppose \(\delta\) is a limit and we've defined \((a_{\eta})_{\eta \in \omega^{\leq \alpha}}\) for all \(\alpha < \delta\).  Given any \(\eta \in \omega^{\delta}\), we may, by saturation, find an element \(b \in \bigcap_{\alpha < \delta} f^{-1}_{\alpha \delta}(a_{\eta | \alpha})\).  Then we can set \(a_{\eta} = b\).  This gives \((a_{\eta})_{\eta \in \omega^{\leq \delta}}\) and completes the construction.  

Given \(\alpha < \kappa\), let \(\varphi_{\alpha}(x;y)\) be the formula \(p_{\alpha}(x) = y\).  For any \(\eta \in \omega^{\kappa}\), \(\{\varphi_{\alpha}(x;a_{\eta | \alpha}) : \alpha < \kappa\}\) is consistent and, for all \(\nu \in \omega^{<\kappa}\), \(\{\varphi_{l(\nu)+1}(x;a_{\nu \frown \langle i \rangle}) : i < \omega \}\) is 2-inconsistent.  We have thus exhibited a cdt-pattern of height \(\kappa\) so \(\kappa_{\text{cdt}}(T^{*}_{\kappa,f}) \geq \kappa^{+}\). 

By Lemma \ref{stable} and Fact \ref{easy inequalities}, we have $\kappa_{\mathrm{cdt}}(T^{*}_{\kappa,f}) \leq \kappa^{+}$, so we have the desired equality.  
\end{proof}

\subsection{Case 1: \(\kappa_{\text{inp}} = \kappa^{+}\)} 

In this subsection, we first show how to produce a homogeneous set of size $\kappa$ for $f$ from an $\mathrm{inp}$-pattern of a very particular form.  Then, using rectification, we observe that every $\mathrm{inp}$-pattern of height $\kappa$ gives rise to one of this particular form.  Together, these will allow us to calculate an upper bound on $\kappa_{\mathrm{inp}}(T^{*}_{\kappa,f})$ when the coloring $f$ is chosen to have no homogeneous set of size $\kappa$.   

%
%
\begin{lem}\label{case2}
Suppose we are given a collection \((\beta_{\alpha, i})_{\alpha < \kappa, i < 2}\) of ordinals smaller than $\kappa$ such that if \(\alpha < \alpha' < \kappa\), then $\beta_{\alpha, 0} \leq \beta_{\alpha,1}$, $\beta_{\alpha',0} \leq \beta_{\alpha',1}$, $\beta_{\alpha,0} \leq \beta_{\alpha',0}$ and $\beta_{\alpha,1} < \beta_{\alpha',1}$.  Suppose that there is a mutually indiscernible array \((c_{\alpha, k})_{\alpha < \kappa, k < \omega}\) such that, with \(\varphi_{\alpha}(x;y_{\alpha})\) defined by \((f_{\beta_{\alpha,0}\beta_{\alpha,1}} \circ p_{\beta_{\alpha,1}})(x) = y_{\alpha}\), \((\varphi_{\alpha}(x;y_{\alpha}) : \alpha < \kappa)\), \((c_{\alpha, k})_{\alpha < \kappa, k < \omega}\) forms an inp-pattern of height \(\kappa\).  Then for all pairs \(\alpha < \alpha'\), \(f(\{\beta_{\alpha,1}, \beta_{\alpha',1}\}) = 1\).  
\end{lem}

\begin{proof}
If \(\alpha < \alpha'\) and \(f(\{\beta_{\alpha,1}, \beta_{\alpha',1}\}) = 0\), then $p_{\beta_{\alpha,1}}(x) = (f_{\beta_{\alpha,1} \beta_{\alpha',1}} \circ p_{\beta_{\alpha',1}})(x)$ for any \(x\) with \(p_{\beta_{\alpha,1}}(x) \neq x\) and \(p_{\beta_{\alpha',1}}(x) \neq x\), and hence 
\begin{eqnarray*}
(f_{\beta_{\alpha,0}\beta_{\alpha,1}} \circ p_{\beta_{\alpha,1}})(x) &=& (f_{\beta_{\alpha,0}\beta_{\alpha,1}} \circ f_{\beta_{\alpha,1}\beta_{\alpha',1}} \circ p_{\beta_{\alpha',1}})(x) \\
&=& (f_{\beta_{\alpha,0}\beta_{\alpha',1}} \circ p_{\beta_{\alpha',1}})(x) \\
&=& (f_{\beta_{\alpha,0}, \beta_{\alpha',0}} \circ f_{\beta_{\alpha',0}\beta_{\alpha',1}} \circ p_{\beta_{\alpha',1}})(x).
\end{eqnarray*}
Consequently, 
\[
\{(f_{\beta_{\alpha,0}\beta_{\alpha,1}} \circ p_{\beta_{\alpha,1}})(x) = c_{\alpha, k'}, (f_{\beta_{\alpha',0}\beta_{\alpha',1}}\circ p_{\beta_{\alpha',1}})(x) = c_{\alpha',k}\}
\]
is consistent only if \(c_{\alpha, k'} = f_{\beta_{\alpha,0}\beta_{\alpha',0}}(c_{\alpha',k})\).  Because for all $\xi < \kappa$, $(c_{\xi,i})_{i < \omega}$ is indiscernible and, by the definition of an $\mathrm{inp}$-pattern, $\{\varphi_{\xi}(x;c_{\xi,i}) : i < \omega\}$ is inconsistent, we know that $c_{\xi,l} \neq c_{\xi,l'}$ for $l \neq l'$.  Fix any $k<\omega$.  We have shown there is a unique $k'$ such that   
\[
\{(f_{\beta_{\alpha,0}\beta_{\alpha,1}} \circ p_{\beta_{\alpha,1}})(x) = c_{\alpha, k'}, (f_{\beta_{\alpha',0}\beta_{\alpha',1}}\circ p_{\beta_{\alpha',1}})(x) = c_{\alpha',k}\}
\]
is consistent.  By the definition of an $\mathrm{inp}$-pattern, given any function $g: \kappa \to \omega$, 
$$
\{\varphi_{\alpha}(x;c_{\alpha,g(\alpha)}) : \alpha < \kappa\}
$$
is consistent and so, in particular, the set
\[
\{(f_{\beta_{\alpha,0}\beta_{\alpha,1}} \circ p_{\beta_{\alpha,1}})(x) = c_{\alpha, g(\alpha)}, (f_{\beta_{\alpha',0}\beta_{\alpha',1}}\circ p_{\beta_{\alpha',1}})(x) = c_{\alpha',g(\alpha')}\}
\]
is consistent.  Choosing $g(\alpha') = k$ and $g(\alpha) \neq k'$, we obtain a contradiction.  
\end{proof}

%
%

For the remainder of this subsection, we will assume there is an inp-pattern of height $\kappa$ modulo $T$.  By Lemma \ref{tending}, it follows there is a \emph{rectified} inp-pattern of height $\kappa$ and, by \cite[Corollary 2.9]{ChernikovNTP2} and Remark \ref{same number of variables}, we may assume that this is witnessed by a rectified inp-pattern in a single free variable.  Hence, for the rest of this subsesction, we will fix a rectified inp-pattern \((\overline{\varphi},\overline{k},(a_{\alpha,i})_{\alpha < \kappa, i < \omega},\overline{w})$ and we will assume that each $\varphi_{\alpha}(x;y_{\alpha})$ enumerated in $\overline{\varphi}$ has \(l(x) = 1\).  Recall the associated \(\Delta\)-system is denoted \(\overline{w} = (w_{\alpha}: \alpha < \kappa)\) with root \(r = \{\zeta_{i} : i < n\}\) and \(w_{\alpha} \setminus r = v_{\alpha} = \{\beta_{\alpha, j} : j <m\}\), where the enumerations are increasing.

\begin{lem} \label{ino}
For all \(\alpha < \kappa\), \(\varphi_{\alpha}(x;y_{\alpha}) \vdash x \in O\).  
\end{lem}

\begin{proof}
First, note that we may assume that there is a predicate \(Q \in \{O, P_{\zeta_{i}} : i < n\}\) such that $\varphi_{\alpha}(x;y_{\alpha}) \vdash x \in Q$ for all $\alpha < \kappa$. If not, using that the $w_{\alpha}$'s form a $\Delta$-system, that every formula $\varphi_{\alpha}(x;y_{\alpha})$ is complete, and that $\varphi_{\alpha}(x;a_{\alpha,i})$ is consistent with $\varphi_{\beta}(x;a_{\beta,j})$ whenever $\alpha \neq \beta$, there would be some $\alpha < \kappa$ such that $\varphi_{\alpha}(x;y_{\alpha})$ implies that $x$ is not contained in any predicate of $L_{w_{\alpha}}$.  By Lemma \ref{no equalities}(1), we know each $\varphi_{\alpha}(x;a_{\alpha,i})$ is non-algebraic, so, in this case it is easy to check that $\{\varphi_{\alpha}(x;a_{\alpha, i}) : i < \omega\}$ is consistent, contradicting the definition of inp-pattern.  So we must show that  \(\varphi_{\alpha}(x;y_{\alpha}) \vdash P_{\zeta_{i}}\) for some \(i < n\) is impossible.  

Suppose not and fix $i_{*} < n$ so that $\varphi_{\alpha}(x;y_{\alpha}) \vdash x \in P_{\zeta_{i_{*}}}$ for some $\alpha < \kappa$.  Note that it follows that $\varphi_{\alpha}(x;y_{\alpha}) \vdash x \in P_{\zeta_{i_{*}}}$ for \emph{all} $\alpha < \kappa$ as each $\varphi_{\alpha}$ is a complete $L_{w_{\alpha}}$-formula, the predicate $P_{\zeta_{i_{*}}}$ is in every $L_{w_{\alpha}}$, and columns in the $\mathrm{inp}$-pattern are consistent.  Write each tuple in the array \(a_{\alpha,i}\) as \(a_{\alpha,i} = (b_{\alpha, i}, c_{\alpha, i}, d_{\alpha, i}, e_{\alpha, i})\) where the elements of \(b_{\alpha, i}\) are in \(O\), the elements of \(c_{\alpha, i}\) are in predicates indexed by the root \(\bigcup_{i < n} P_{\zeta_{i}}\), the elements of \(d_{\alpha,i}\) are in predicates whose index is in \(\bigcup_{j < m} P_{\beta_{\alpha,j}}\), and the elements of \(e_{\alpha,i}\) are not in any predicate of \(L_{w_{\alpha}}\).  By completeness, quantifier-elimination, as well as Lemmas \ref{no equalities}(1) and \ref{nice functions on P}, each \(\varphi_{\alpha}(x;a_{\alpha,i})\) is equivalent to the conjunction of the following:
\begin{enumerate}
\item \(x \in P_{\zeta_{i_{*}}}\) 
\item \(x \neq (a_{\alpha,i})_{l}\) for all \(l < l(a_{\alpha,i})\)  
\item \((f_{\gamma \zeta_{i_{*}}}(x) = (c_{\alpha,i})_{l})^{t_{\gamma,l}}\) for all \(l < l(c_{\alpha,i})\) and \(\gamma \in w_{\alpha}\) less than \(\zeta_{i_{*}}\) and some \(t_{\gamma,l} \in \{0,1\}\).  
\end{enumerate}
For each \(k < i_{*}\), let \(\gamma_{k}\) be the least ordinal \(<\kappa\) such that \(\varphi_{\gamma_{k}}(x;a_{\gamma_{k},0}) \vdash f_{\alpha_{k}\alpha_{i_{*}}}(x) = c\) for some \(c \in c_{\gamma_{k},0}\) and \(0\) if there is no such.  Let \(\gamma = \max\{\gamma_{k} : k < i_{*}\}\). We claim that \(\{\varphi_{\gamma+1}(x;a_{\gamma+1,j}) : j < \omega\}\) is consistent.    Note that any equality of the form $f_{\zeta_{k}\zeta_{i_{*}}}(x) = c$ implied by $\varphi_{\gamma+1}(x;a_{\gamma+1,j})$ is implied by $\varphi_{\gamma_{k}}(x;a_{\gamma_{k},0})$ by indiscernibility and the fact that, for all \(j < \omega\),
\[
\{\varphi_{\gamma_{k}}(x;a_{\gamma_{k},0}) ,\varphi_{\gamma+1}(x;a_{\gamma+1,j})\}
\]
is consistent.  Additionally, any inequality of the form \(f_{\zeta_{k}\zeta_{i_{*}}}(x) \neq c\) implied by \(\varphi_{\gamma+1}(x;a_{\gamma+1,j})\) is compatible with \(\{\varphi_{\alpha}(x;a_{\alpha,0}) : \alpha \leq \gamma\}\).  Choosing a realization \(b \models \{\varphi_{\alpha}(x;a_{\alpha,0}) : \alpha \leq \gamma\}\) satisfying every inequality of the form \(f_{\zeta_{k}\zeta_{i_*}}(x) \neq c\) implied by the \(\varphi_{\gamma+1}(x;a_{\gamma+1,j})\) yields a realization of \(\{\varphi_{\gamma+1}(x;a_{\gamma+1,j}) : j < \omega\}\), by the description of $\varphi_{\gamma+1}(x;a_{\gamma+1,j})$ as a conjunction given above.  This contradicts the definition of inp-pattern.  
\end{proof}

\begin{prop}\label{inpcomputation}
If $\kappa_{\mathrm{inp}}(T^{*}_{\kappa,f}) = \kappa^{+}$, then there is a subset \(H \subseteq \kappa\) with \(|H| = \kappa\) such that \(f\) is constant on \([H]^{2}\).   
\end{prop}

\begin{proof}
Recall that the hypothesis $\kappa_{\mathrm{inp}}(T^{*}_{\kappa,f}) = \kappa^{+}$ allowed us to fix a rectified inp-pattern \((\overline{\varphi},\overline{k},(a_{\alpha,i})_{\alpha < \kappa, i < \omega},\overline{w})$ with the property that each $\varphi_{\alpha}(x;y_{\alpha})$ enumerated in $\overline{\varphi}$ has \(l(x) = 1\).

By completeness and Lemma \ref{ino}, we know that, for each \(\alpha < \kappa\), \(\varphi_{\alpha}(x;y)\vdash x \in O\).  Then by quantifier-elimination, completeness, and Lemmas \ref{no equalities}(2) and \ref{nice functions on O}, for each $\alpha < \kappa$, $\varphi_{\alpha}(x;a_{\alpha,0})$ is equivalent to the conjunction of the following: 
\begin{enumerate}
\item \(x \in O\)
\item \(x \neq (a_{\alpha,0})_{l}\) for all \(l < l(a_{\alpha,0})\) 
\item \((p_{\gamma}(x) = x)^{t^{0}_{\gamma}}\) for $\gamma \in w_{\alpha}$ and some $t^{0}_{\gamma} \in \{0,1\}$.
\item The values of the \(p_{\gamma}\) and how they descend in the tree:
\begin{enumerate}
\item $((f_{\delta \gamma} \circ p_{\gamma})(x) = (a_{\alpha,0})_{l})^{t^{1}_{l,\delta,\gamma}}$ for $l < l(a_{\alpha,0})$, $\delta \leq \gamma$ in $w_{\alpha}$, and some $t^{1}_{l,\delta,\gamma} \in \{0,1\}$. 
\item \(((f_{\delta \gamma} \circ p_{\gamma})(x) = (f_{\delta \gamma'} \circ p_{\gamma'})(x))^{t^{2}_{\delta,\gamma,\gamma'}}\) for \(\delta, \gamma, \gamma' \in w_{\alpha}\) with \(\delta \leq \gamma < \gamma'\), for some $t^{2}_{\delta,\gamma,\gamma'} \in \{0,1\}$.  
\end{enumerate}
\end{enumerate}


\textbf{Claim:  }Given $\alpha < \kappa$, there are $\epsilon_{\alpha} \leq \epsilon'_{\alpha} \in w_{\alpha}$ and pairwise distinct $c_{\alpha,k} \in a_{\alpha,k}$ such that, for all $k < \omega$, $\varphi_{\alpha}(x;a_{\alpha,k}) \vdash (f_{\epsilon_{\alpha} \epsilon_{\alpha}'} \circ p_{\epsilon'_{\alpha}})(x) = c_{\alpha,k}$.  

\emph{Proof of claim:}  Suppose not.  Then, by the description of $\varphi_{\alpha}(x;a_{\alpha,k})$ given above, the following set of formulas 
\[
\{\varphi_{\alpha}(x;a_{\alpha,k}) : k < \omega\}
\]
is equivalent to a finite number of equations common to each instance \(\varphi_{\alpha}(x;a_{\alpha,k})\) and an infinite collection of inequations.  Then, it is easy to see then that \(\{\varphi_{\alpha}(x;a_{\alpha,k}) : k < \omega\}\) is consistent, contradicting the definition of an inp-pattern.  This proves the claim.

Note that, by the pigeonhole principle, we may assume that either (i) $\epsilon_{\alpha}, \epsilon_{\alpha}' \in r$ for all $\alpha < \kappa$, (ii) $\epsilon_{\alpha} \in r$, $\epsilon'_{\alpha} \in v_{\alpha}$ for all $\alpha < \kappa$, or (iii) $\epsilon_{\alpha},\epsilon'_{\alpha} \in v_{\alpha}$ for all $\alpha < \kappa$.  

Case (i) is impossible:  as the root \(r = \{\zeta_{i} : i < n\}\) is finite and the all 0's path is consistent, we can find an ordinal \(\gamma < \kappa\) such that for all \(\alpha < \kappa\), if there is a \(c \in a_{\alpha,0}\) such that \(\varphi_{\alpha}(x;a_{\alpha,0}) \vdash (f_{\zeta_{i}\zeta_{i'}} \circ p_{\zeta_{i'}})(x) = c\) for some \(i \leq i' < n\), then there is some \(\alpha' < \gamma\) such that \(\varphi_{\alpha'}(x;a_{\alpha',0}) \vdash (f_{\zeta_{i}\zeta_{i'}} \circ p_{\zeta_{i'}})(x) = c\).  Hence, by indiscernibility, the equality $(f_{\epsilon_{\gamma} \epsilon_{\gamma}'} \circ p_{\epsilon'_{\gamma}})(x) = c_{\gamma,k}$ implied by $\varphi_{\gamma}(x;a_{\gamma,k})$ must also be implied by $\varphi_{\alpha}(x;a_{\alpha,0})$ for some $\alpha < \gamma$.  Since $\{\varphi_{\alpha}(x;a_{\alpha,0}), \varphi_{\gamma}(x;a_{\gamma,k})\}$ is consistent for all $k < \omega$, this is impossible because the tuples in $(c_{\alpha,k})_{k < \omega}$ are pairwise distinct.  

Now we consider cases (ii) and (iii). Again by the pigeonhole principle, we may assume that if we are in case (ii), then $\epsilon_{\alpha}$ is constant for all $\alpha$.  Then by rectification, we know that, in either case (ii) or (iii), when $\alpha < \alpha'$, $\epsilon_{\alpha} \leq \epsilon_{\alpha'}$ and $\epsilon'_{\alpha} < \epsilon'_{\alpha'}$.  Because for all $\alpha < \kappa$, the $c_{\alpha,k}$ are pairwise distinct and $k$ varies, the set of formulas
$$
\{(f_{\epsilon_{\alpha}\epsilon'_{\alpha}} \circ p_{\epsilon'_{\alpha}})(x) = c_{\alpha,k} : k < \omega\}
$$
is $2$-inconsistent.  Moreover, if $g : \kappa \to \omega$ is a function, the partial type 
$$
\{ (f_{\epsilon_{\alpha}\epsilon_{\alpha'}} \circ p_{\epsilon'_{\alpha}})(x) = c_{\alpha,g(\alpha)} : \alpha < \kappa\}
$$
is implied by $\{\varphi_{\alpha}(x;a_{\alpha,g(\alpha)}) : \alpha < \kappa\}$ and is therefore consistent.  It follows that $((f_{\epsilon_{\alpha}\epsilon'_{\alpha}} \circ p_{\epsilon'_{\alpha}})(x) = y_{\alpha})_{\alpha < \kappa}, (c_{\alpha,k})_{\alpha < \kappa, k < \omega}$ is an inp-pattern with $k_{\alpha} = 2$ for all $\alpha < \kappa$.  By Lemma \ref{case2}, $f(\{\epsilon'_{\alpha},\epsilon'_{\alpha'}\}) = 1$ for all $\alpha < \alpha'$.  Therefore $H = \{\epsilon'_{\alpha} : \alpha < \kappa\}$ is a homogeneous set for $f$.  
\end{proof}

\subsection{Case 2:  \(\kappa_{\text{sct}} = \kappa^{+}\)}


In this subsection, we show that if $\kappa_{sct}(T^{*}_{\kappa,f}) = \kappa^{+}$ then $f$ satisfies a homogeneity property inconsistent with $f$ being a strong coloring.  In particular, we will show that if this homogeneity property fails, then for any putative sct-pattern of height $\kappa$, there are two incomparable elements in $\omega^{<\kappa}$ which index compatible formulas, contradicting the inconsistency condition in the definition of an sct-pattern.  This step is accomplished by relating consistency of the relevant formulas to an amalgamation problem in finite structures.  The following lemma describes the relevant amalgamation problem:

\begin{lem}\label{consistency}
Suppose we are given the following:
\begin{itemize} 
\item Finite sets \(w, w' \subset \kappa\) with \(w \cap w' = v\) such that for all \(\alpha \in v\), \(\beta \in w \setminus v\), \(\gamma \in w' \setminus v\), we have \(\alpha < \beta < \gamma\) and \(f(\{\beta, \gamma\}) = 1\).
\item Structures \(A \in \mathbb{K}_{w \cup w'}\), \(B = \langle d,A \rangle^{B}_{L_{w}} \in \mathbb{K}_{w}\), \(C = \langle e, A \rangle^{C}_{L_{w'}} \in \mathbb{K}_{w'}\) satisfying the following:
\begin{enumerate}
\item The tuples $d,e$ are contained in $O \cup \bigcup_{\alpha\in v} P_{\alpha}$.
\item 	The map sending \(d \mapsto e\) induces an isomorphism of $L_{v}$-structures over \(A\) between $B = \langle d,A \rangle^{B}_{L_{v}}$ and $C = \langle e, A \rangle_{L_{v}}^{C}$.
\end{enumerate}
\end{itemize}
Then there is \(D = \langle f,A \rangle^{D}_{L_{w \cup w'}} \in \mathbb{K}_{w \cup w'}\) extending \(A\) such that \(l(f)= l(d) = l(e)\) and \(\langle f, A \rangle^{D}_{L_{w}} \cong B\) over \(A\) and \(\langle f,A \rangle^{D}_{L_{w'}} \cong C\) over \(A\) via the isomorphisms over \(A\) sending \(f \mapsto d\) and \(f \mapsto e\), respectively.  
\end{lem}

\begin{proof}
Let \(f\) be a tuple of formal elements with \(l(f) = l(d)\)(\(=l(e)\)) with \(L_{w}\) and \(L_{w'}\) interpreted so that \(\langle f,A \rangle_{L_{w}}\) extends \(A\) and is $L_{w}$-isomorphic over \(A\) to \(B\), so that \(\langle f,A \rangle_{L_{w'}}\) extends \(A\) and is $L_{w'}$-isomorphic over \(A\) to \(C\), and so that $\langle f,A \rangle_{L_{w}}$ and $\langle f,A \rangle_{L_{w'}}$ are disjoint over $A \cup \{f\}$.  Let $\gamma$ be the least element of $w' \setminus v$ and define \(D\) to have underlying set
\[
\langle f,A \rangle_{L_{w}} \cup \langle f,A \rangle_{L_{w'}} \cup \{*_{\alpha,c} : \alpha \in w\setminus v, c \in P_{\gamma}^{\langle f,A \rangle_{L_{w'}}} \setminus P_{\gamma}^{A} \}.
\]
We must give \(D\) an \(L_{w \cup w'}\)-structure.  The main task is to give elements at the levels of the tree indexed by $\alpha \in w' \setminus v$ ancestors at the levels of $w \setminus v$ and the new formal elements $*_{\alpha,c}$ will play this role.   

Interpret the predicates on $D$ by setting $O^{D} = O^{\langle f,A \rangle_{L_{w}}	} = O^{\langle f,A \rangle_{L_{w'}}}$ and, additionally, 
$$
P^{D}_{\alpha} = \left\{
\begin{matrix}
P_{\alpha}^{\langle f,A \rangle_{L_{w'}}} & \text{ if } \alpha \in w' \setminus v \\
P_{\alpha}^{\langle f,A \rangle_{L_{w}}} \cup \{*_{\alpha,c}: c \in P_{\gamma}^{\langle f,A \rangle_{L_{w'}}} \setminus P_{\gamma}^{A}\} & \text{ if } \alpha \in w \setminus v \\
P_{\alpha}^{\langle f, A \rangle_{L_{w}}} \cup P^{\langle f,A \rangle_{L_{w'}}}_{\alpha} & \text{ if } \alpha \in v.
 \end{matrix}
\right.
$$

For each of the function symbols $f^{D}_{\alpha \beta}$, we are forced to interpret $f^{D}_{\alpha \beta}$ to be the identity on the complement of $P_{\beta}^{D}$ in $D$, so it suffices to specify the interpretation on $P^{D}_{\beta}$. Given \(\alpha \in w\setminus v\) and $c \in P_{\gamma}^{\langle f,A \rangle_{L_{w'}}} \setminus P_{\gamma}^{A}$, interpret \(f^{D}_{\alpha \gamma}(c) = {*}_{\alpha,c}\) and for any \(\beta \in w'\setminus v\), define \(f_{\alpha \beta}^{D} = f_{\alpha \gamma}^{D} \circ f_{\gamma \beta}^{D}\) on \(P_{\beta}^{D}\).  If \(\alpha \in w\setminus v\) and \(\xi \in v\), interpret \(f^{D}_{\xi \alpha}\) so that $f^{D}_{\xi \alpha}|_{P^{\langle f,A\rangle_{L_{w}}}_{\alpha}} = f^{\langle f,A \rangle_{L_{w}}}_{\xi \alpha}|_{P^{\langle f,A\rangle_{L_{w}}}_{\alpha}}$ and \(f^{D}_{\xi \alpha}(*_{\alpha,c}) = f^{D}_{\xi \gamma}(c)\).  If $\alpha < \beta$ are both from $w \setminus v$, we likewise define $f^{D}_{\alpha \beta}$ so that $f^{D}_{\alpha \beta}|_{P^{\langle f,A \rangle_{L_{w}}}_{\beta}} = f_{\alpha \beta}^{\langle f, A \rangle_{L_{w}}}$ and $f^{D}_{\alpha \beta}(*_{\beta,c}) = *_{\alpha,c}$.  

It remains to define the interpretation of $f_{\alpha \beta}^{D}$ when $\alpha <\beta$ are from $(w \cup w')$ and $\alpha,\beta \notin w \setminus v$.  If $\beta \in w'$, then we can only set $f^{D}_{\alpha \beta}|_{P^{D}_{\beta}} = f^{\langle f,A \rangle_{L_{w'}}}_{\alpha \beta}|_{P^{D}_{\beta}}$, since $P^{D}_{\beta} = P_{\beta}^{\langle f,A \rangle_{L_{w'}}}$.  If $\beta \in v$, then we set $f^{D}_{\alpha \beta}|_{P^{D}_{\beta}} = f^{\langle f, A \rangle_{L_{w}}}_{\alpha \beta}|_{P^{\langle f, A \rangle_{L_{w}}}_{\beta}} \cup f^{\langle f, A \rangle_{L_{w'}}}_{\alpha \beta}|_{P^{\langle f, A \rangle_{L_{w'}}}_{\beta}}$

Finally, interpret each function of the form \(p_{\beta}\) for \(\beta \in w\) to restrict to $p_{\beta}^{\langle f,A \rangle_{L_{w}}}$ and to be the identity on the complemement of $\langle f,A \rangle_{L_{w}}$ and likewise for $\beta \in w'$ (note that these definitions agree for $\alpha \in w \cap w' = v$).  This completes the definition of the \(L_{w \cup w'}\)-structure on \(D\).  It is clear from construction that $D$ is an $L_{w \cup w'}$-extension of $A$, an $L_{w}$-extension of $\langle f,A \rangle_{L_{w}}$, and an $L_{w'}$-extension of $\langle f,A \rangle_{L_{w'}}$.  

Now we must check that \(D \in \mathbb{K}_{w \cup w'}\).  It is easy to check that axioms \((1)-(3)\) are satisfied in \(D\).  As \(f(\{\alpha, \beta\}) = 1\) for all \(\alpha \in w \setminus v, \beta \in w' \setminus v\), the only possible counterexample to axiom (4) can occur when \(\xi \in v\), \(\beta \in (w \cup w') \setminus v\) and \(f(\{\xi, \beta\})=0\).  As the formal elements \(*_{\alpha, c}\) are not in the image of \(O\) under the \(p_{\alpha}\), it follows that a counterexample to axiom (4) must come from a counter-example either in \(B\) or \(C\), which is impossible.  So \(D \in \mathbb{K}_{w \cup w'}\), which completes the proof.  
\end{proof}

\begin{lem}\label{rootandobject}
Suppose \(((\varphi_{\alpha}(x;y_{\alpha}))_{\alpha < \kappa},(a_{\eta})_{\eta \in \omega^{<\kappa}},\overline{w})\) is a rectified sct-pattern such that \(l(x)\) is minimal among sct-patterns of height \(\kappa\).  Then for all \(\alpha < \kappa\), \(\varphi_{\alpha}(x;y_{\alpha}) \vdash (x)_{l} \in O \cup \bigcup_{i < n} P_{\zeta_{i}}\) for all \(l < l(x)\), that is, every formula in the pattern implies that every variable $(x)_{l}$ is in $O$ or a predicate indexed by the root of the associated $\Delta$-system.  
\end{lem}

\begin{proof}
Suppose not.  First, consider the case that for some \(l < l(x)\) and all \(\alpha < \kappa\), $\varphi_{\alpha}(x;y_{\alpha}) \vdash (x)_{l} \not\in O \cup \bigcup_{i < n} P_{\zeta_{i}} \cup \bigcup_{j < m} P_{\beta_{\alpha,j}}$, then the only relations that \(\varphi_{\alpha}(x;y_{\alpha})\) can assert between \((x)_{l}\) and the elements of \(y_{\alpha}\) and the other elements of \(x\) are equalities and inequalities.  By Lemma \ref{no equalities}(2), we know that $\varphi_{\alpha}(x;y_{\alpha})$ proves no equalities between elements of $x$ and the element of $y_{\alpha}$ so it can only prove inequalties between $(x)_{l}$ and $y_{\alpha}$, but it is easy to see that this allows us to find an sct-pattern in fewer variables, contradicting minimality (or if $l(x) = 1$ the definition of an sct-pattern).  

Secondly, consider the case that there is some \(\alpha < \kappa\) and \(j < m\) such that $\varphi_{\alpha}(x;y_{\alpha}) \vdash (x)_{l} \in P_{\beta_{\alpha, j}}$ and therefore, for all \(\alpha' \neq \alpha\), \(\varphi_{\alpha'}(x;y_{\alpha'})\) implies that \((x)_{l}\) is not in any of the unary predicates of \(L_{w_{\alpha'}}\), as \(\beta_{\alpha,j}\) is outside the root of the \(\Delta\)-system.  So restricting the given pattern to the formulas \((\varphi_{\alpha'}(x;y_{\alpha'}) : \alpha' < \kappa, \alpha' \neq \alpha)\) yields a rectified sct-pattern of height $\kappa$ which falls into the first case considered, a contradiction.  As these are the only cases, we conclude.  
\end{proof}

\begin{prop}\label{sctcomputation}
If $\kappa_{\mathrm{sct}}(T^{*}_{\kappa,f}) = \kappa^{+}$, then there is \(\gamma\) such that for any \(\alpha,\alpha'\) with \(\gamma <\alpha < \alpha'<\kappa\) there is \(\xi \in v_{\alpha}, \zeta \in v_{\alpha'}\) such that \(f(\{\xi, \zeta\}) = 0\).  
\end{prop}

\begin{proof}
Suppose not.   Recall that by Lemma \ref{tending} and Remark \ref{same number of variables}, if there is an sct-pattern of height $\kappa$ in $k$-free variables, there is a sct-pattern in $k$ free variables which is also rectified.  It follows we may fix a rectified sct-pattern \(((\varphi_{\alpha}(x;y_{\alpha}))_{\alpha < \kappa},(a_{\eta})_{\eta \in \omega^{<\kappa}},\overline{w})\) such that \(l(x)\) is minimal among sct-patterns of height \(\kappa\).  By Lemma \ref{rootandobject}, we know that up to a relabeling of the variables, there is a \(k \leq l(x)\) such that, for all $l < k$, \(\varphi_{\alpha}(x;y_{\alpha}) \vdash (x)_{l} \in P_{\zeta_{i(l)}}\) for some $i(l)<n$ and \(\varphi_{\alpha}(x;y_{\alpha}) \vdash (x)_{l} \in O\) for \(l \geq k\).  

For each $\alpha < \kappa$, let $\varphi'_{\alpha}(x)$ be a complete $L_{w_{\alpha}}$-formula, without parameters, in the variables $x$ implied by $\varphi_{\alpha}(x;y_{\alpha})$ (which is unique up to logical equivalence, since $\varphi_{\alpha}(x;y_{\alpha})$ was assumed to be a complete $L_{w_{\alpha}}$-formula).  Clearly we have, for all $l < k$, \(\varphi'_{\alpha}(x) \vdash (x)_{l} \in P_{\zeta_{i(l)}}\) and \(\varphi'_{\alpha}(x) \vdash (x)_{l} \in O\) for \(l \geq k\), since these are formulas without parameters in $L_{r} \subseteq L_{w_{\alpha}}$. Since all the symbols in the language are unary, it is easy to see from quantifier-elimination that for each $\alpha < \kappa$ and $\eta \in \omega^{\alpha}$, $\varphi_{\alpha}(x;a_{\eta})$ is equivalent to a conjunction of the following:
\begin{enumerate}
\item $\varphi'_{\alpha}(x)$.
\item $(x)_{l} \neq (a_{\eta})_{i}$ for $l < l(x)$ and $i < l(a_{\eta})$ (using the minimality of $l(x)$).
\item $(f_{\delta \zeta_{i(l)}}((x)_{l}) = (a_{\eta})_{i})^{t^{0}_{\delta,l,i}}$ for $l < k$, $\delta\in r$ with $\delta < \zeta_{i(l)}$, and $i < l(a_{\eta})$, and for some $t^{0}_{\delta,l,i} \in \{0,1\}$. 
\item $((f_{\delta \xi} \circ p_{\xi})((x)_{l}) = (a_{\eta})_{i})^{t^{1}_{\delta, \xi,l,i}}$ for $\delta \leq \xi$ from $r$, $k \leq l < l(x)$, and $i < l(a_{\eta})$, and for some $t^{1}_{\delta, \xi,l,i} \in \{0,1\}$. 
\item $((f_{\delta \xi} \circ p_{\xi})((x)_{l}) = (a_{\eta})_{i})^{t^{2}_{\delta, \xi,l,i}}$ for $\delta \leq \xi$ from $w_{\alpha}$, $\xi \in v_{\alpha}$, $k \leq l < l(x)$, and $i < l(a_{\eta})$, and for some $t^{2}_{\delta, \xi,l,i} \in \{0,1\}$.
\end{enumerate}
Choose \(\gamma < \kappa\) so that if \(\alpha < \kappa\) and \(\varphi_{\alpha}(x;a_{0^{\alpha}})$ implies a positive instance of one of the equalities in (3) and (4), then this is implied by \(\varphi_{\alpha'}(x;a_{0^{\alpha'}})\) for some \(\alpha' < \gamma\) (possible as the root is finite).   

By assumption, there are \(\alpha, \alpha'\) with \(\gamma < \alpha < \alpha' < \kappa\) such that \(f(\{\xi, \zeta\}) = 1\) for all \(\xi \in v_{\alpha}, \zeta \in v_{\alpha'}\).  Choose \(\eta \in \omega^{\alpha}\), \(\nu \in \omega^{\alpha'}\) both extending $0^{\gamma}$ such that \(\eta \perp \nu\).  Let \(A = \langle a_{\eta}, a_{\nu} \rangle_{L_{w_{\alpha} \cup w_{\alpha'}}}\) be the finite \(L_{w_{\alpha} \cup w_{\alpha'}}\)-structure generated by \(a_{\eta}\) and \(a_{\nu}\).  Pick $d \models \{\varphi_{\delta}(x;a_{0^{\delta}}): \delta \leq \gamma\} \cup \{\varphi_{\alpha}(x;a_{\eta})\}$ and $e \models \{\varphi_{\delta}(x;a_{0^{\delta}}) : \delta \leq \gamma\} \cup \{\varphi_{\alpha'}(x;a_{\nu})\}$.  By the choice of $\gamma$, the $s$-indiscernibility of $(a_{\eta})_{\eta \in \omega^{<\kappa}}$, and quantifier-elimination and the observation above, we have \(\text{tp}_{L_{r}}(d/A) = \text{tp}_{L_{r}}(e/A)\).  Let \(B = \langle d,A \rangle_{L_{w_{\alpha}}}\) and \(C = \langle e,A \rangle_{L_{w_{\alpha'}}}\).  By Lemma \ref{consistency}, there is a \(D \in \mathbb{K}_{w_{\alpha} \cup w_{\alpha'}}\) such that \(D = \langle g, A \rangle^{D}_{L_{w_{\alpha} \cup w_{\alpha'}}}\) such that \(l(g) = l(d) = l(e)\) and \(\langle g,A \rangle_{L_{w_{\alpha}}} \cong B\) over \(A\) and \(\langle g, A \rangle_{L_{w_{\alpha'}}} \cong C\) over \(A\).  Using the extension property to embed $D$ in $\mathbb{M}$ over $A$, it follows that in \(\mathbb{M}\), \(g \models \{\varphi_{\alpha}(x;a_{\eta}), \varphi_{\alpha'}(x;a_{\nu})\}\), contradicting the definition of sct-pattern.  This completes the proof.  
\end{proof}

\subsection{Conclusion}

\begin{thm} \label{first main theorem}
There is a stable theory \(T\) such that \(\kappa_{\text{cdt}}(T) \neq \kappa_{\text{sct}}(T) + \kappa_{\text{inp}}(T)\).  Moreover, it is consistent with ZFC that for every regular uncountable \(\kappa\), there is a stable theory \(T\) with \(|T| = \kappa\) and \(\kappa_{\text{cdt}}(T) > \kappa_{\text{sct}}(T) + \kappa_{\text{inp}}(T)\).  
\end{thm}

\begin{proof}
If $\kappa$ is regular and uncountable satisfying $\text{Pr}_{1}(\kappa,\kappa,2,\aleph_{0})$, then choose \(f: [\kappa]^{2} \to 2\) witnessing \(\text{Pr}_{1}(\kappa,\kappa,2,\aleph_{0})\).  There can be no homogeneous set of size \(\kappa\) for \(f\), since given any $\{x_{\alpha} : \alpha < \kappa\} \subseteq \kappa$, enumerated in increasing order, we obtain a pairwise disjoint family of finite sets $(v_{\alpha})_{\alpha < \kappa}$ defined by $v_{\alpha} = \{x_{\alpha}\}$ and $\mathrm{Pr}_{1}(\kappa,\kappa,2,\aleph_{0})$ implies that for each color $i \in \{0,1\}$, there are $\alpha < \alpha'$ such that $f(\{x_{\alpha},x_{\alpha'}\}) =i$. Moreover, $\mathrm{Pr}_{1}(\kappa,\kappa,2,\aleph_{0})$ implies directly that there can be no collection \((v_{\alpha} : \alpha < \kappa)\) of disjoint finite sets such that, given \(\alpha < \alpha' < \kappa\), there are \(\xi \in v_{\alpha}, \zeta \in v_{\alpha'}\) such that \(f(\{\xi, \zeta\}) = 0\).  Let \(T = T^{*}_{\kappa, f}\).  This theory is stable by Lemma \ref{stable}. Additionally, \(\kappa_{\mathrm{cdt}}(T) = \kappa^{+}\), by Proposition \ref{cdtcomputation}, but \(\kappa_{\text{sct}}(T) < \kappa^{+}\) and \(\kappa_{\text{inp}}(T) < \kappa^{+}\) by Proposition \ref{sctcomputation} and Proposition \ref{inpcomputation} respectively.  By Fact \ref{ShelahPr} and Observation \ref{monotonicity}, \(\text{Pr}_{1}(\lambda^{++}, \lambda^{++}, 2, \aleph_{0})\) holds for any regular uncountable $\lambda$.  Then \(T = T^{*}_{\kappa,f}\) gives the desired theory, for \(\kappa = \lambda^{++}\) and any \(f\) witnessing \(\text{Pr}_{1}(\lambda^{++}, \lambda^{++}, 2, \aleph_{0})\).  For the ``moreover'' clause, note that ZFC is equiconsistent with ZFC + GCH + ``there are no inaccessible cardinals" (if \(V \models \text{ZFC}\) has a strongly inaccessible in it, replace \(V\) by \(V_{\kappa}\) for \(\kappa\) the least such, then consider \(L\) in \(V\)) which entails that every regular uncountable cardinal is a successor.  By Theorem \ref{GalvinPr} this implies that \(\text{Pr}_{1}(\kappa, \kappa,2, \aleph_{0})\) holds for all regular uncountable cardinals \(\kappa\), which completes the proof.  \end{proof}
%
%
%
\begin{rem}
In \cite[Theorem 3.1]{ArtemNick}, it was proved that \(\kappa_{\text{cdt}}(T) = \kappa_{\text{inp}}(T) + \kappa_{\text{sct}}(T)\) for any countable theory \(T\).  The above theorem shows that in a certain sense, this result is best possible.  
\end{rem}

\begin{rem}
It would be interesting to know if for $\kappa$ strongly inaccessible, there is a theory $T$ with $\kappa_{\mathrm{cdt}}(T) = \kappa^{+} > \kappa_{\text{inp}}(T) + \kappa_{\text{sct}}(T)$.  
\end{rem}

\section{Compactness of ultrapowers}\label{compactness}

In this section we study the decay of saturation in regular ultrapowers. We say an ultrafilter $\mathcal{D}$ on $I$ is \emph{regular} if there is a collection of sets $\{X_{\alpha} : \alpha < |I|\} \subset \mathcal{D}$ such that for all $t \in I$, the set $\{\alpha : t \in X_{\alpha}\}$ is finite and $\mathcal{D}$ is \emph{uniform} if all sets in $\mathcal{D}$ have cardinality $|I|$.  Recall that a model $M$ is called \emph{$\lambda$-compact} if every (partial) type over $M$ of cardinality less than $\lambda$ is realized in $M$.  In the case that the language has size at most $\lambda$, the notions of $\lambda$-compactness and $\lambda$-saturation are equivalent but they may differ if the cardinality of the language exceeds $\lambda$, since, in this case, types over sets of parameters of size less than $\lambda$ may still contain more than $\lambda$ many formulas, in general.  Given a theory $T$, we start with a regular uniform ultrafilter $\mathcal{D}$ on $\lambda$ and a $\lambda^{++}$-saturated model $M \models T$.  We then consider whether the ultrapower $M^{\lambda}/\mathcal{D}$ is $\lambda^{++}$-compact.  Shelah has shown \cite[Theorem VI.4.7]{shelah1990classification} that if $T$ is not simple, then in this situation $M^{\lambda}/\mathcal{D}$ will not be $\lambda^{++}$-compact and asked whether an analogous result holds for theories $T$ with $\kappa_{\text{inp}}(T) > \lambda^{+}$.  We will show by direct construction that $\kappa_{\text{inp}}(T) > \lambda^{+}$ does not suffice but, by modifying an argument due to Malliaris and Shelah \cite[Claim 7.5]{Malliaris:2012aa}, $\kappa_{\text{sct}}(T) > \lambda^{+}$ is sufficient to obtain a decay in compactness, by levaraging the finite square principles of Kennedy and Shelah \cite{kennedyshelah}.  

\subsection{A counterexample}

Fix $\kappa$ a regular uncountable cardinal.  Let $L'_{\kappa} = \langle O, P_{\alpha},p_{\alpha} : \alpha < \kappa \rangle$ be a language where $O$ and each $P_{\alpha}$ is a unary predicate and each $p_{\alpha}$ is a unary function.  Define a theory $T'_{\kappa}$ to be the universal theory with the following as axioms:
\begin{enumerate}
\item $O$ and the $(P_{\alpha})_{\alpha < \kappa}$ are pairwise disjoint.
\item For all $\alpha < \kappa$, $p_{\alpha}$ is a function such that $(\forall x \in O)[p_{\alpha}(x) \in P_{\alpha}]$ and $(\forall x \not\in O)[p_{\alpha}(x) = x]$.
\end{enumerate}
Given a finite set $w \subset \kappa$, define $L'_{w} = \langle O,P_{\alpha}, p_{\alpha} : \alpha \in w \rangle$.  Let $\mathbb{K}'_{w}$ denote the class of finite models of $T'_{\kappa} \upharpoonright L'_{w}$.  

\begin{lem}
Suppose $w \subset \kappa$ is finite.  Then $\mathbb{K}'_{w}$ is a Fra\"iss\'e class.
\end{lem}

\begin{proof}
The axioms of $T'_{\kappa}\upharpoonright L_{w}$ are universal so HP is clear.  As we allow the empty structure to be a model, JEP follows from AP.  For AP, we reduce to the case where $A,B,C \in \mathbb{K}'_{w}$, $A$ is a substructure of both $B$ and $C$ and $B \cap C = A$.  Because all the functions in the language are unary, we may define an $L'_{w}$-structure $D$ on $B \cup C$ by taking unions of the relations and functions as interpreted on $B$ and $C$.  It is easy to see that $D \in \mathbb{K}'_{w}$, so we are done.  
\end{proof}

By Fra\"iss\'e theory, for each finite $w \subset \kappa$, there is a unique countable ultrahomogeneous $L'_{w}$-structure with age $\mathbb{K}'_{w}$.  Let $T^{\dag}_{w}$ denote its theory.  

We remark that the theory $T^{\dag}_{w}$ is almost a reduct of $T^{*}_{w}$ considered in the previous sections, with the difference that the functions $p_{\alpha}$ are partial in $T^{*}_{w}$ and total in $T^{\dag}_{w}$.  One can easily check that $T^{\dag}_{w}$ is interpretable in $T^{*}_{w}$ for $w$ finite, interpreting $O$ by $\bigwedge_{\alpha \in w} \text{dom}(p_{\alpha})$.  Since this interpretation is not uniform in $w$, we will still need to rapidly repeat the same steps in the analysis above to show that the $T^{\dag}_{w}$ are coherent.  

\begin{lem}
Suppose $v$ and $w$ are finite sets with $w \subset v \subset \kappa$.  Then $T^{\dag}_{w} \subset T^{\dag}_{v}$.  
\end{lem}

\begin{proof}
By induction, it suffices to consider the case when $v = w \cup \{\gamma\}$ for some $\gamma \in \kappa \setminus w$.  By Fact \ref{KPTredux}, we must show (1) that $A \in \mathbb{K}'_{w}$ if and only if there is $D \in \mathbb{K}'_{v}$ such that $A$ is an $L'_{w}$-substructure of $D \upharpoonright L'_{w}$ and (2) that whenever $A,B \in \mathbb{K}'_{w}$, $\pi : A \to B$ is an $L'_{w}$-embedding, and $C \in \mathbb{K}'_{v}$ satisfies $C = \langle A \rangle^{C}_{L'_{v}}$ then there is $D \in \mathbb{K}'_{v}$ such that $B$ is an $L'_{w}$-substructure of $D \upharpoonright L'_{w}$ and $\pi$ extends to an $L'_{v}$-embedding $\tilde{\pi} : C \to D$.  

For (1), it is clear from definitions that if $D \in \mathbb{K}'_{v}$ then $D \upharpoonright L'_{w} \in \mathbb{K}'_{w}$.  Given $A \in \mathbb{K}'_{w}$, we may construct a suitable $L'_{v}$-structure $D$ as follows.  If $O^{A} = \emptyset$, we may simply expand $A$ to $D$ by setting $P_{\gamma}^{D} = \emptyset$ and this trivially satisfies the required axioms. So we will assume $O^{A}$ is non-empty and let the underlying set of $D$ be $A \cup \{*\}$.  We interpret the predicates of $L'_{w}$ to have the same interpretation as on $A$, and we interpret the functions of $L'_{w}$ so that their restriction to $A$ are their interpretations on $A$ and so that the functions are the identity on $*$.  We additionally set $P^{D}_{\gamma} = \{*\}$ and $p^{D}_{\gamma}$ to be the identity on the complement of $O^{D}$ ($=O^{A}$) and the constant function with value $*$ on $O^{D}$.  Clearly $D \in \mathbb{K}'_{w}$, $D = \langle A \rangle_{L'_{v}}$, and $A$ is an $L'_{w}$-substructure of $D \upharpoonright L'_{w}$. 

For (2), suppose $A,B \in \mathbb{K}'_{w}$, $\pi : A \to B$ is an embedding, and $C \in \mathbb{K}'_{v}$ satisfies $C = \langle A \rangle^{C}_{L'_{v}}$.  The requirement that $C = \langle A \rangle^{C}_{L'_{v}}$ entails that any points of $C \setminus A$ lie in $P_{\gamma}^{C}$.  In particular, $O^{A} = O^{C}$ and we may use this notation interchangeably.  Let $E = O^{B} \setminus \pi(O^{A})$, so that we may write $O^{B} = \pi(O^{A}) \sqcup E$.  Define an $L'_{v}$-structure $D$ whose underlying set is $B \cup P_{\gamma}(A) \cup \{*_{e} : e \in E\}$.  Interpret the predicates of $L'_{w}$ on $D$ to have the same interpretation as on $B$ and interpret the functions of $L'_{w}$ so that they agree with their interpretations on $B$ and are the identity on the complement of $B$. Then define $P_{\gamma}(D) = P_{\gamma}(A) \cup \{*_{e} : e \in E\}$ and interpret $p^{D}_{\gamma}$ by 
$$
p_{\gamma}^{D}(x) = \left\{
\begin{matrix}
p_{\gamma}^{C}(a) & \text{ if } x = \pi(a) \\
*_{x} & \text{ if } x \not\in \pi(O^{A}).
\end{matrix}
\right.
$$
Clearly $D \in \mathbb{K}'_{v}$.  Extend $\pi$ to a map $\tilde{\pi}: C \to D$ by defining $\pi$ to be the identity on $P_{\gamma}(C)$.  We claim $\tilde{\pi}$ is an $L'_{v}$-embedding:  note that for all $x \in O^{C}$, $p_{\gamma}^{D}(\tilde{\pi}(x)) = p_{\gamma}^{C}(x) = \tilde{\pi}(p_{\gamma}^{C}(x))$ and $\tilde{\pi}$ obviously respects all other structure from $L'_{w}$ as $\pi$ is an $L'_{w}$-embedding.  
\end{proof}

Define the theory $T^{\dag}_{\kappa}$ to be the union of $T^{\dag}_{w}$ for all finite $w \subset \kappa$.  This is a complete $L'_{\kappa}$-theory with quantifier elimination, as these properties are inherited from the $T^{\dag}_{w}$.  Fix a monster $\mathbb{M} \models T_{\kappa}^{\dag}$ and work there. 
%
%
%
\begin{prop}\label{bound}
The theory $T^{\dag}_{\kappa}$ is stable and $\kappa_{inp}(T^{\dag}_{\kappa}) = \kappa^{+}$.
\end{prop}

\begin{proof}
For each $\alpha < \kappa$, choose for each $\beta < \omega$ $a_{\alpha,\beta} \in P_{\alpha}(\mathbb{M})$ such that $\beta \neq \beta'$ implies $a_{\alpha,\beta} \neq a_{\alpha,\beta'}$.  It is easy to check that, for all functions $g: \kappa \to \omega$, $\{p_{\alpha}(x) = a_{\alpha,g(\alpha)} : \alpha < \kappa\}$ is consistent and, for all $\alpha < \kappa$, $\{p_{\alpha}(x) = a_{\alpha, \beta} : \beta < \omega\}$ is $2$-inconsistent by the injectivity of the sequence $(a_{\alpha,\beta})_{\beta < \omega}$.  Setting $k_{\alpha} = 2$ for all $\alpha$, we see that $(p_{\alpha}(x) = y_{\alpha} : \alpha < \kappa)$, $(a_{\alpha, \beta})_{\alpha < \kappa, \beta < \omega}$, and $(k_{\alpha})_{\alpha < \kappa}$ forms an inp-pattern of height $\kappa$ so $\kappa_{\text{inp}}(T_{\kappa}^{\dag}) \geq \kappa^{+}$.  The stability of $T^{\dag}_{\kappa}$ follows from an argument identical to Lemma \ref{stable} which, by Fact \ref{easy inequalities}, gives the upper bound $\kappa_{\text{inp}}(T_{\kappa}^{\dag}) \leq \kappa^{+}$.
\end{proof}

\begin{prop}\label{saturation}
Suppose $\mathcal{D}$ is an ultrafilter on $\lambda$, $\kappa = \lambda^{+}$, and $M \models T^{\dag}_{\kappa}$ is $\lambda^{++}$-saturated.  Then $M^{\lambda}/\mathcal{D}$ is $\lambda^{++}$-saturated.  
\end{prop}

\begin{proof}
Suppose $A \subseteq M^{\lambda}/\mathcal{D}$, $|A| = \kappa = \lambda^{+}$.  To show that any $q(x) \in S^{1}(A)$ is realized, we have three cases to consider:
\begin{enumerate}
\item $q(x) \vdash x \in P_{\alpha}$ for some $\alpha < \kappa$
\item $q(x) \vdash x \not\in O$ and $q(x) \vdash x\not\in P_{\alpha}$ for all $\alpha < \kappa$
\item $q(x) \vdash x \in O$.  
\end{enumerate}
It suffices to consider $q$ non-algebraic and $A = \text{dcl}(A)$.  In case (1), $q(x)$ is implied by 
$\{P_{\alpha}(x)\} \cup \{x \neq a : a \in A\}$ and in case (2), $q(x)$ is implied by $\{\neg O(x) \wedge \neg P_{\alpha}(x) : \alpha < \kappa\} \cup \{x \neq a : a \in A\}$.  To realize $q(x)$ in case (1), for each $t \in \lambda$, choose some $b_{t} \in P_{\alpha}(M)$ such that $b_{t} \neq a[t]$ for all $a \in A$, which is possible by the $\lambda^{++}$-saturation of $M$ and the fact that $|A| = \lambda^{+}$.  Let $b = \langle b_{t} \rangle_{t \in \lambda}/\mathcal{D}$.  By $\L$o\'{s}'s theorem, $b \models q$.  Realizing $q$ in case (2) is entirely similar. 

So now we show how to handle case (3).  Fix some complete type $q(x) \in S_{1}(A)$ such that $q(x) \vdash x \in O$.  First, we note that by possibly growing $A$ by $\kappa$ many elements, we may assume that there is a sequence $(c_{\alpha})_{\alpha < \kappa}$ from $A$ so that $q$ is equivalent to the following:
$$
\{x \in O\} \cup \{x \neq a : a \in O(A)\} \cup \{p_{\alpha}(x) = c_{\alpha}\},
$$ 
This follows from the fact that, for each $\alpha < \kappa$, either $q(x) \vdash p_{\alpha}(x) = c_{\alpha}$ for some $c_{\alpha}$, or it only proves inequations of this form.  In the latter case, we can choose some element $c_{\alpha} \in P_{\alpha}(M^{\lambda}/\mathcal{D})$ not in $A$ (possible by case (1) above) and extend $q(x)$ by adding the formula $p_{\alpha}(x) = c_{\alpha}$, which will then imply all inequations of the form $p_{\alpha}(x) \neq a$ for any $a \in A$, and this clearly remains finitely satisfiable.  So now given $q$ in the form described above, let $X_{t} = \{\alpha < \kappa : M \models P_{\alpha}(c_{\alpha}[t])\}$ for each $t \in \lambda$.  Let $q_{t}(x)$ denote the following set of formulas over $M$: 
$$
q_{t}(x) = \{x \in O\} \cup \{x \neq a[t]: a \in O(A)\} \cup \{p_{\alpha}(x) = c_{\alpha}[t] : \alpha \in X_{t}\}.
$$
By construction, if $\alpha \neq \alpha' \in X_{t}$ then $M \models P_{\alpha}(c_{\alpha}[t]) \wedge P_{\alpha'}(c_{\alpha'}[t])$ so this set of formulas is consistent and over a parameter set from $M$ of size at most $\kappa$, hence realized by some $b_{t} \in M$.  Let $b = \langle b_{t} \rangle_{t \in \lambda}/\mathcal{D}$ and let $J_{\alpha}$ be defined by $J_{\alpha} = \{t \in \lambda : M \models P_{\alpha}(c_{\alpha}[t])\}$.  Note that, for $t < \lambda$ and $\alpha < \kappa$, $t \in J_{\alpha}$ if and only if $\alpha \in X_{t}$.  As $q(x)$ is a consistent set of formulas, $J_{\alpha} \in \mathcal{D}$ and, by construction, $J_{\alpha} \subseteq \{t \in \lambda : M \models p_{\alpha}(b_{t}) = c_{\alpha}[t]\}$ so $M^{\lambda}/\mathcal{D} \models p_{\alpha}(b) = c_{\alpha}$.  It is obvious that $b$ satisfies all of the other formulas of $q$ so we are done.  
\end{proof}

\begin{cor} \label{second main theorem}
Suppose $T$ is a complete theory, $|I| = \lambda$, $\mathcal{D}$ on $I$ is a ultrafilter, and $M \models T$ is a $\lambda^{++}$-saturated model of $T$.  The condition that $\kappa_{\text{inp}}(T) > |I|^{+}$ is, in general, not sufficient to guarantee that $M^{I}/\mathcal{D}$ is not $\lambda^{++}$-compact.  In particular, by Fact \ref{easy inequalities}(2), the condition that $\kappa_{\text{cdt}}(T) > |I|^{+}$ is not sufficient to guarantee that $M^{I}/\mathcal{D}$ is not $\lambda^{++}$-compact.
\end{cor}

\begin{proof}
Given $\lambda$, $I$ with $|I| = \lambda$, and an ultrafilter $\mathcal{D}$ on $I$, choose any $\lambda^{++}$-saturated model of $T^{\dag}_{\lambda^{+}}$.  By Lemma \ref{bound}, $\kappa_{\text{cdt}}(T^{\dag}_{\lambda^{+}}) \geq \kappa_{\text{inp}}(T^{\dag}_{\lambda^{+}})  = \lambda^{++} > |I|^{+}$, but, by Proposition \ref{saturation}, $M^{I}/\mathcal{D}$ is $\lambda^{++}$-saturated and hence $\lambda^{++}$-compact.  
\end{proof}

\subsection{Loss of saturation from large sct-patterns}

If $T$ is not simple, then it has either the tree property of the first kind or the second kind\textemdash Shelah argues in \cite[Theorem VI.4.7]{shelah1990classification} by demonstrating that either property results in a decay of saturation with an argument tailored to each property.  The preceding section demonstrates that the analogy between TP$_{2}$ and $\kappa_{\text{inp}}(T) > |I|^{+}$ breaks down, but we show that the analogy between TP$_{1}$ and $\kappa_{\text{sct}}(T) > |I|^{+}$ survives, assuming some set theory.  The argument below is a straightforward adaptation of the argument of \cite[Claim 8.5]{Malliaris:2012aa}.

Recall that if $T$ is a theory with a distinguished predicate $P$ and $\kappa < \lambda$ are infinite cardinals, then the theory $T$ is said to \emph{admit} $(\lambda, \kappa)$ if there is a model $M \models T$ with $|M| = \lambda$ and $|P^{M}| = \kappa$.  The notation $\langle \kappa, \lambda \rangle \to \langle \kappa', \lambda' \rangle$ stands for the assertion that any theory in a countable language that admits $(\lambda,\kappa)$ also admits $(\lambda',\kappa')$.  Chang's two-cardinal theorem asserts that if $\lambda = \lambda^{<\lambda}$ then $\langle \aleph_{0},\aleph_{1} \rangle \to \langle \lambda, \lambda^{+}\rangle$ (see, e.g., \cite[Theorem 7.2.7]{chang1990model}\textemdash the statement given here follows from the proof).    

\begin{fact}\label{square} \cite[Lemma 4]{kennedyshelah}
Suppose $\mathcal{D}$ is a regular uniform ultrafilter on $\lambda$ and $\langle \aleph_{0},\aleph_{1} \rangle \to \langle \lambda, \lambda^{+}\rangle$.  There is an array of sets $\langle u_{t,\alpha} : t < \lambda, \alpha < \lambda^{+} \rangle$ satisfying the following properties:
\begin{enumerate}
\item $u_{t, \alpha} \subseteq \alpha$ 
\item $|u_{t, \alpha}| < \lambda$
\item $\alpha \in u_{t, \beta} \implies u_{t, \beta} \cap \alpha = u_{t, \alpha}$
\item if $u \subseteq \lambda^{+}$, $|u| < \aleph_{0}$ then $\{t < \lambda : (\exists \alpha)(u \subseteq u_{t, \alpha})\} \in \mathcal{D}$.
\end{enumerate}
\end{fact}

\begin{thm} \label{second main theorem part 2}
Suppose $|I| = \lambda$ and $\langle \aleph_{0},\aleph_{1} \rangle \to \langle \lambda, \lambda^{+}\rangle$.  Suppose $\kappa_{\text{sct}}(T) > |I|^{+}$, $M$ is an $|I|^{++}$-saturated model of $T$ and $\mathcal{D}$ is a regular ultrafilter over $I$.  Then $M^{I}/\mathcal{D}$ is not $|I|^{++}$-compact.  
\end{thm}

\begin{proof}
Let $(\varphi_{\alpha}(x;y_{\alpha}) : \alpha < \lambda^{+})$, $(a_{\eta})_{\eta \in \lambda^{<\lambda^{+}}}$ be an sct-pattern.  We may assume $l(y_{\alpha}) = k$ for all $\alpha < \lambda^{+}$.  Let $\langle u_{t,\alpha} : t < \lambda, \alpha < \lambda^{+} \rangle$ be given as by Fact \ref{square}.  We may consider the tree $(\lambda^{+})^{<\lambda}$ as the set of sequences of elements of $\lambda^{+}$ of length $<\lambda$ ordered by extension and then, for each $t < \lambda$ and $\alpha < \lambda^{+}$, we can define $\eta_{t,\alpha} \in (\lambda^{+})^{<\lambda}$ to be the sequence that enumerates $u_{t,\alpha} \cup \{\alpha\}$ in increasing order.  Note that if $\alpha < \beta$, then, because $\alpha \in u_{t,\beta}$ implies $u_{t,\beta} \cap \alpha = u_{t,\alpha}$, we have $\eta_{t,\alpha} \vartriangleleft \eta_{t,\beta} \iff \alpha \in u_{t,\beta}$.  

For each $\alpha < \lambda^{+}$ we thus have an element $c_{\alpha} \in M^{\lambda}/\mathcal{D}$ given by $c_{\alpha} = \langle c_{\alpha}[t]: t < \lambda \rangle /\mathcal{D}$ where $c_{\alpha}[t] = a_{\eta_{t, \alpha}} \in M$.  

\textbf{Claim:} $p(x):= \{\varphi_{\alpha}(x;c_{\alpha}) : \alpha < \lambda^{+} \}$ is consistent.  

\emph{Proof of claim:} Fix any finite $u \subseteq \lambda^{+}$.  If for some $t < \lambda$ and $\alpha < \lambda^{+}$, we have $u \subseteq u_{t,\alpha}$ then $\{\eta_{t,\beta} : \beta \in u\} \subseteq \{\eta_{t,\beta} : \beta \in u_{t,\alpha}\}$ which is contained in a path, hence $\{\varphi_{\beta}(x;c_{\beta}[t]) : \beta \in u\} = \{\varphi_{\beta}(x;a_{\eta_{t,\beta}}) : \beta \in u\}$ is consistent by definition of an sct-pattern.  We know $\{t < \lambda : (\exists \alpha)(u \subseteq u_{t,\alpha})\} \in \mathcal{D}$ so the claim follows by $\L$o\'{s}'s theorem and compactness.\qed  

Suppose $b = \langle b[t] \rangle_{t \in \lambda} / \mathcal{D}$ is a realization of $p$ in $M^{\lambda}/\mathcal{D}$.  For each $\alpha < \lambda^{+}$ define $J_{\alpha} = \{t < \lambda : M \models \varphi_{\alpha}(b[t],c_{\alpha}[t])\} \in \mathcal{D}$.  For each $\alpha$, pick $t_{\alpha} \in J_{\alpha}$.  The map $\alpha \mapsto t_{\alpha}$ is regressive on the stationary set of $\alpha$ with $\lambda \leq \alpha < \lambda^{+}$.  By Fodor's lemma, there's some $t_{*}$ such that the set $S = \{\alpha < \lambda^{+} : t_{\alpha} = t_{*}\}$ is stationary.  Therefore $p_{*}(x) = \{\varphi_{\alpha}(x;a_{\eta_{t_{*},\alpha}}) : \alpha \in S\}$ is a consistent partial type in $M$ so $\{\eta_{t_{*},\alpha} : \alpha \in S\}$ is contained in a path, by definition of sct-pattern.  Choose an $\alpha \in S$ so that $|S \cap \alpha| = \lambda$.  Then, by choice of the $\eta_{t,\alpha}$, we have $\beta \in S \cap \alpha$ implies $\eta_{t_{*},\beta} \unlhd \eta_{t_{*},\alpha}$ and therefore $\beta \in u_{t_{*},\alpha}$.  This shows $|u_{t_{*},\alpha}| \geq \lambda$, a contradiction.  
\end{proof}

\bibliographystyle{alpha}
\bibliography{ms.bib}{}

\end{document}